\documentclass{article}

\usepackage{amsmath, amssymb}
\usepackage[amsthm, thref, hyperref,thmmarks, amsmath]{ntheorem}
\usepackage{mathtools} 
\usepackage{accents}

\usepackage{geometry}
\usepackage{hhline}
\usepackage{float}
\usepackage{multicol}

\usepackage{graphicx}
\usepackage{tikz}
\usetikzlibrary{calc, patterns, positioning}
\usepackage{wrapfig}

\usepackage{cancel}
\usepackage[labelfont=bf]{caption}

\usepackage{hyperref}
\usepackage[backend=biber, hyperref=true, style=numeric, sorting=nty]{biblatex}
\newtheorem{Lemma}{Lemma}[section]
\newtheorem{Theorem}{Theorem}[section]

\theoremstyle{definition}

\newtheorem{Assumption}{Assumption}[section]
\newtheorem{Algorithm}{Algorithm}[section]
\newtheorem*{Remark}{Remark}

\theoremsymbol{\hfill\ensuremath{\pmb\square}}
\newtheorem*{Proof}{Proof}

\newcommand{\RR}{\mathbb{R}}

\newcommand{\PP}{\mathbb{P}}
\newcommand{\CC}{\mathbb{C}}
\newcommand{\VV}{\mathbb{V}}
\newcommand{\TT}{\mathbb{T}}
\newcommand{\EE}{\mathbb{E}}

\newcommand{\wh}[1]{\widehat #1}
\newcommand{\x}{\mathbf{x}}

\newcommand*\Div{\mbox{div}}
\newcommand*\Curl{\mbox{curl}}

\newcommand{\cl}[1]{\overline{#1}}

\newcommand{\nullspace}{\mbox{nullspace }}
\newcommand{\image}{\mbox{image }}
\newcommand{\rank}{\mbox{rank }}
\newcommand{\npw}{n^{(pw)}}
\newcommand{\hnorm}{|\hspace{-.6px}|\hspace{-.6px}|}
\newcommand{\kk}{\mathrm{k}}

\usetikzlibrary{decorations.pathreplacing,calligraphy}

\title{Multi-resolution Isogeometric Analysis -- Efficient adaptivity utilizing the multi-patch structure}
\author{Stefan Takacs\footnote{stefan.takacs@jku.at, Johannes Kepler University Linz, Austria.}, Stefan Tyoler\footnote{stefan.tyoler@ricam.oeaw.ac.at, RICAM, Austrian Academy of Sciences, Austria.\\This work was supported by the Austrian Science Fund (FWF): P33956}}
\date{May 2024}

\begin{document}

\maketitle

\begin{abstract}
Isogeometric Analysis (IgA) is a spline based approach to the numerical solution of partial differential equations. There are two major issues that IgA was designed to address. The first issue is the exact representation of domains stemming from Computer Aided Design (CAD) software. In practice, this can be realized only with multi-patch IgA, often in combination with trimming or similar techniques. The second issue is the realization of high-order discretizations (by increasing the spline degree) with numbers of degrees of freedom comparable to low-order methods. High-order methods can deliver their full potential only if the solution to be approximated is sufficiently smooth; otherwise, adaptive methods are required. In the last decades, a zoo of local refinement strategies for splines has been developed. The authors think that many of these approaches are a burden to implement efficiently and impede the utilization of recent advances that rely on tensor-product splines, e.g., concerning matrix assembly and preconditioning. The implementation seems to be particularly cumbersome in the context of multi-patch IgA. Our approach is to moderately increase the number of patches and to utilize different grid sizes on different patches. This allows reusing the existing code bases, recovers the convergence rates of other adaptive approaches and increases the number of degrees of freedom only marginally.
\end{abstract}

\section{Introduction}

When Isogeometric Analysis (IgA) was originally proposed almost two decades ago in the seminal paper~\cite{Hughes2005} by Hughes, Cottrell and Bazilevs, it was supposed to solve two major issues. The first issue was that a standard realization of a Finite Element Analysis (FEA) requires the computational domain to be meshed. IgA aims to bridge the gap between Computer Aided Design (CAD) and FEA, allowing to directly use the geometry description from design for analysis. The second goal of IgA was to more easily allow the construction of smooth basis functions. This is of relevance since it reduces the number of degrees of freedom and the number of spurious eigenmodes and it allows the construction of conforming discretizations for fourth and higher order problems. IgA has gained substantial interest since it was first proposed, cf. the book~\cite{Cottrell2009} and references therein.

Technically speaking, in IgA one uses B-Splines or Non-Uniform Rational B-Splines (NURBS) as ansatz functions for a Galerkin or collocation discretization of the given boundary value problem. Starting from univariate B-Splines, their extension to two and more dimensions is typically realized be setting up tensor products for the unit square or the unit cube. The basis functions on the computational domain are then defined via the pull-back principle, utilizing a parameterization of the computational domain that may stem from the CAD model. Since the parameterization is required to be continuous, only simply connected domains can be parameterized with a single function. In practice, the overall computational domain is decomposed into subdomains, usually called patches, which are parameterized separately. This approach retains the tensor-product structure and the high smoothness of splines locally. Concerning the global smoothness conditions, one usually only imposes conditions that are necessary in order to obtain a conforming discretization; for second order differential equations, whose conforming discretization is usually posed in the Sobolev space $H^1$, only continuity is imposed. Certainly, having only reduced smoothness across the interfaces increases the number of degrees of freedom. However, this is typically negligible; we will elaborate on this in Section~\ref{sec:3}.

One of the main benefits of a high-order discretization is their superior approximation power, compared to low-order discretizations. For a discretization with polynomial or spline degree $p$, the approximation error in the standard $H^1$ Sobolev norm decays like $h^p$. Certainly, this is only true if the function to be approximated, typically the solution, is smooth enough, specifically it needs to be in the Sobolev space $H^{p+1}$. So, the higher the chosen degree is, the more smoothness is usually desired for full approximation power.

For many problems, where the desired smoothness is not available globally, it is usually available on large portions of the computational domain. Only close to certain features, like corners or changes in say material parameters, the smoothness of the solution is reduced. This motivates the use of adaptive algorithms that allow for local refinement close to these features. Adaptivity is usually based on the \emph{solve---estimate---mark---refine} loop. After solving the problem on the current discretization, an error estimator is used in order to determine the areas where the mesh is refined. Here, we restrict ourselves to residual based error estimators, cf.~\cite{Verfuerth}. In a next step, certain areas are marked as to be refined, usually with D\"orfler marking~\cite{Doerfler1996,Pfeiler2020}. Finally, a finer discretization is set up.

In standard finite element contexts, refinement is usually done using simple bisection methods. Since one usually wants to avoid hanging nodes, there are many strategies to avoid them, like red-greed based refinement techniques or nearest vertex bisection, cf.~\cite{Verfuerth,Binev2004} to name just a few. Also, in spline contexts, several approaches have been proposed over the last few decades. Among the most prominent examples are hierarchical B-splines (HB) and truncated hierarchical B-splines (THB), see~\cite{Kraft1997,Giannelli2012,Giannelli2016} and others, T-splines, see~\cite{Sederberg2004}, locally refined (LR) splines, see~\cite{Johannessen2014} and many others.

There are several targets one wants to achieve with such spline constructions. First of all these splines must offer the desired (local) approximation power, meaning the function space needs to be rich enough. Certainly, this should be possible without introducing too many degrees of freedom. Moreover, the approach needs to be simple enough to be implemented efficiently, including in a multi-patch context.

Most local spline constructions known from literature are formulated for single patch discretizations; often, their generalization to the multi-patch case is a non-trivial task. So, we introduce an approach that is formulated for the multi-patch case straight from the beginning. In order to keep the method simple, we use the multi-patch structure (which we need anyway for the representation of the geometry) also for adaptivity.

In our approach, we keep the local tensor-product structure within each patch. In order to allow for local refinement, we might decide to refine the B-splines only on a few patches or to split patches into smaller patches, where the grid size could be specified for each of the smaller patches separately. Compared to more standard approaches, this leads to discretizations that are not ``fully matching'' across the interfaces. One possibility to handle such cases is the use of mortar or discontinuous Galerkin approaches~\cite{LangerToulopoulos2015,LangerMantzaflaris2015}. In our case, the discretizations on the interfaces between any two patches are nested. This allows the use of standard conforming Galerkin discretizations. We show that it is possible to construct bases for the resulting spaces and to provide the desired approximation error estimates.

Besides the fact that our approach is inherently using the multi-patch structure, one of its advantages is that it preserves the local tensor product structure of the problem. So, approaches that rely on this structure, like for matrix assembling, cf.
\cite{Calabro2017,Mantzaflaris2017}
and others, or for preconditioning, cf.
\cite{Hofreither2017,Sangalli2016} and others,
can still be used. It is worth stressing again that this approach of splitting patches slightly increases the number of degrees of freedom (since the functions are only $C^0$ smooth at the interfaces of the patches), however this is not significant and is outweighed by the benefits of our approach.

This paper is organized as follows. In Section~\ref{sec:2}, we introduce the adaptive refinement strategy that is used throughout this paper. A mathematical specification of the non-conforming multi-patch geometries produced by the adaptive refinement strategy is specified in Section~\ref{sec:3}. The construction of a basis for the overall function space and its properties are discussed in Section~\ref{sec:4}. In Section~\ref{sec:5}, we provide approximation error estimates. Extensive numerical experiments are presented in Section~\ref{sec:6}. Conclusions are drawn in Section~\ref{sec:7}.

\section{Isogeometric Galerkin method}\label{sec:2}

We are discussing our approach for a simple elliptic model problem with 
Dirichlet boundary conditions in two dimensions. Aspects of extending the approach to three dimensions are addressed at the end of this paper. We denote the computational domain by $\Omega \subset \mathbb R^2$ and assume it to be open, connected and bounded with a Lipschitz continuous boundary $\partial \Omega$. Given a strictly positive diffusion parameter $\nu \in L^\infty(\Omega)$ and a right-hand-side function $f\in L^2(\Omega)$, the problem reads in variational form as follows.
\[
		\mbox{Find } u\in H_0^1(\Omega):\quad
		\underbrace{\int_\Omega \nu \nabla u \cdot \nabla v \mathrm d x}_{\displaystyle a(u,v):=}
		= \underbrace{\int_\Omega f v\, \mathrm d x}_{\displaystyle \ell(v):=}
		\quad
		\forall\; v\in H^1_0(\Omega).
\]
Here and in what follows, $H^k$, $H^k_0$ and $L^2$ are the standard Sobolev and Lebesgue spaces with standard norms $\|\cdot\|_{H^k}$ and $\|\cdot\|_{L^2}$, see~\cite{Adams2003}. The restriction to the boundary is to be understood in the sense of a trace operator.
To discretize this problem, we use multi-patch Isogeometric Analysis.

We assume that the computational domain $\Omega$ is composed of $K$ non-overlapping patches $\Omega_k$, i.e.,
\begin{equation*}
    \cl \Omega = \bigcup_{k=1}^K \cl{\Omega_k} \quad \text{with } \quad
    \Omega_k \cap \Omega_\ell = \emptyset \quad \text{for } k \neq \ell.
\end{equation*}
We assume that the initial configuration is conforming, i.e., the intersection of the closures of any two different patches $\cl{\Omega_k}\cap \cl{\Omega_\ell}$ is either empty, a common corner or a common edge. Here and in what follows, the edge of a patch $\Omega_k$ is the image of one of the four sides of the unit square, like $G_k((0,1)\times\{0\})$. Alike, the corners of a patch are the images of the four corners of the unit square, like $G_k(0,0)$. In the adaptive context, this condition is relaxed to \thref{ass:matching}.

We further assume that each patch $\Omega_k$ is parameterized by a bijective geometry mapping
$
    G_k : \wh\Omega \to \Omega_k
$,
where $\wh\Omega:=(0,1)^2$ is the \textit{parameter domain}. In standard IgA applications, the mappings $G_k$ are B-splines or NURBS, however, we do not require this for our analysis. The parameterizations have to be sufficiently regular in order to satisfy the following assumption.

\begin{Assumption}\label{ass:map}
We assume that there is a uniform constant $C_G>0$ and there are characteristic patch sizes $H_k \ge 0$ and smoothness parameters $r_k > 0$ for each of the patches $k=1,\ldots,K$ such that
\begin{align*}
				&\max_{i\in\{0,\ldots,j\}} \| \partial_x^i \partial_y^{j-i} G_k \|_{L^\infty(\widehat \Omega)} \le C_G H_k^j
				\quad \mbox{for}\quad j = 1,\ldots,r_k+1
                \quad\mbox{and}\quad
				\| (\nabla G_k)^{-1} \|_{L^\infty(\widehat \Omega)} \le C_G H_k^{-1},
\end{align*}
where $\partial_x^i$ and $\partial_y^i$ denotes the $i$-th derivative in direction of $x$ or $y$, respectively. $\nabla G_k$ denotes the Jacobian of $G_k$. Moreover, we assume the same for the restriction of $G_k$ to its parameter lines, i.e.,
\begin{align*}
				&\| \partial^j g \|_{L^\infty(0,1)} \le C_G H_k^j
				\quad \mbox{for}\quad j = 1,\ldots,r_k+1,
				\quad\mbox{and}\quad
				\| |\partial g|^{-1} \|_{L^\infty(0,1)} \le C_G H_k^{-1}
\end{align*}
holds for all $g \in \{ G_k(x,\cdot):x\in[0,1]\}\cup \{ G_k(\cdot,y):y\in[0,1]\}$, where $|\cdot|$ denotes the Euclidean norm.
\end{Assumption}

For each of the patches, we set up independent function spaces of some spline degree $p$. For each spatial direction $\delta \in \{1,2\}$, we assume to have a $p$-open knot vector, i.e., $\Xi^{(k,\delta)} = (\xi_i^{(k,\delta)})_{i=1}^{n^{(k,\delta)}+p+1}$ with
\begin{align*}
				\xi_1^{(k,\delta)}
                &= \xi_2^{(k,\delta)}
                = \cdots
                = \xi_{p+1}^{(k,\delta)}
				\le \xi_{p+2}^{(k,\delta)}
                \le \cdots \le
                \xi_{n^{(k,\delta)}}^{(k,\delta)} \le
				\xi_{n^{(k,\delta)}+1}^{(k,\delta)}
                = \xi_{n^{(k,\delta)}+2}^{(k,\delta)}
                = \cdots = \xi_{n^{(k,\delta)}+p+1}^{(k,\delta)},\\
        \xi_{i}^{(k,\delta)} &< \xi_{i+p}^{(k,\delta)}
        \quad\mbox{for}\quad i=2,\ldots,n^{(k,\delta)}.
\end{align*}

For every knot vector, the associated $B$-spline basis $(B^{(k,\delta)}_{i})_{i=1}^{n^{(k,\delta)}}$ is given by the Cox-de Boor formula, see \cite{deBoor1972}. This basis spans the corresponding B-spline function space $\wh V_h^{(k,\delta)}:= \text{span}\{B^{(k,\delta)}_{i}:i=1\ldots,n^{(k,\delta)}\}$ on the unit interval.

On the parameter domain $\wh \Omega=(0,1)^2$, we define a corresponding tensor-product basis $\wh\Phi^{(k)} :=(\wh \phi_i^{(k)})_{i=1}^{n^{(k)}}$ with $n^{(k)} = n^{(k,1)}n^{(k,2)}$, which is mapped to the physical patch $\Omega_k$ via the \emph{pull-back principle} in order to obtain a basis $\Phi^{(k)} :=(\phi_i^{(k)})_{i=1}^{n^{(k)}}$:
\begin{equation*}
		\wh \phi_{i+(j-1)n^{(k,1)}}(x,y) = B^{(k,1)}_i(x)B^{(k,2)}_j(y)
		\quad\mbox{and}\quad
		\phi_{j}^{(k)}( G_k( \mathbf x ) ) = \wh \phi_j^{(k)}( \mathbf x) .
\end{equation*}
These bases span the spaces
$\wh V_h^{(k)} = \mbox{span} \{\wh \phi_i^{(k)} : i=1,\ldots,n^{(k)} \}$
and
$V_h^{(k)} = \mbox{span} \{\phi_i^{(k)} : i=1,\ldots,n^{(k)} \}$.
Based on such a construction, we denote the grid size $\wh h_k$ associated to the parameter domain and the grid size $h_k$ associated to the physical patch via
\[
	\wh h_k := \max_{\delta\in\{1,2\}} \max_{i\in\{1,\ldots,n^{(k,\delta)}\}} \xi_{i+1}^{(k,\delta)}-\xi_{i}^{(k,\delta)}
	\quad\mbox{and}\quad
	h_k:= H_k \, \wh h_k
\]
and analogously the minimum grid size via
\[
	\wh h_{k,min} := \min_{\delta\in\{1,2\}} \min_{i\in\{1,\ldots,n^{(k,\delta)}:\xi_{i+1}^{(k,\delta)}\ne \xi_{i}^{(k,\delta)}\}} \xi_{i+1}^{(k,\delta)}-\xi_{i}^{(k,\delta)}
	\quad\mbox{and}\quad
	h_{k,min}:= H_k \, \wh h_{k,min}.
\]
The overall function space with grid size $h:=\max\{h_1,\ldots,h_K\}$ is then
\begin{equation}\label{globalsp}
			V_h := \{v \in H^1_0(\Omega) : v|_{\Omega_k} \in V_h^{(k)}\}.
\end{equation}
In order for this to be viable, we need that the trace spaces of $V_h$ associated to the edges between any two patches are rich enough. For standard conforming spaces, this is ensured by the concept of \emph{fully matching discretizations}, which guarantees
\begin{equation}\label{eq:if-agrees}
            V_h\Big|_{\partial\Omega_k \cap \partial\Omega_\ell} = 
			V_h^{(k)}\Big|_{\partial\Omega_k \cap \partial\Omega_\ell}
			= V_h^{(\ell)}\Big|_{\partial\Omega_k \cap \partial\Omega_\ell}
\end{equation}
for any two patches $\Omega_k$ and $\Omega_\ell$ sharing an edge.

The problem is then discretized using the Galerkin principle, leading to the following problem:
\[
		\mbox{Find } u_h\in V_h:\quad
		a(u_h,v_h) = \ell(v_h)
		\quad
		\forall\; v_h\in V_h\subset H^1_0(\Omega).
\]
An a-priori estimate for the overall discretization error is obtained via C\'ea's Lemma. If the solution satisfies the regularity assumption $u\in H^{p+1}(\Omega)$, the discretization error $\|u-u_h\|_{H^1(\Omega)}$ decays like $h^p$. For $p=1$, $H^2$ regularity is desired to obtain the full rates that can be expected for uniformly refined B-spline spaces. Regularity theorems guaranteeing $H^2$-regularity do not apply at reentrant corners and where the diffusion coefficient $\nu$ is discontinuous. For $p\geq 2$, even stronger regularity results are desired, which require even more regularity of the data. In order to recover the optimal rates (with respect to the number of degrees of freedom), adaptive meshing in terms of the grid size is the method of choice.

\section{Adaptivity and non-conforming multi patch discretization}\label{sec:3}

In order to steer adaptive refinement, we use local error estimators to assess the energy error on each of the patches. So, we compute for each patch the value
\begin{equation}\label{res:err:est}
		\eta_k^2:= h_k^2 \|f+\Div(\nu\nabla u)\|_{L^2(\Omega_k)}^2
		+ \sum_{\ell} \frac{h_k}{2} \|\bold{n} \cdot ((\nu \nabla u_h)|_{\Omega_k}-(\nu \nabla u_h)|_{\Omega_\ell})\|_{L^2(\partial\Omega_k \cap \partial\Omega_\ell)}^2,
\end{equation}
where the sum is taken for all indices $\ell$ referring to a patch $\Omega_\ell$ that shares an edge with $\Omega_k$ and $\bold{n}$ is the unit normal vector.
As a next step, we use Dörfler marking
to mark the patches to be refined. As already mentioned, we are interested in methods that preserve the local tensor-product structure on each of the patches. There are several ways this can be achieved. The simplest option would be to refine the grids on all of the marked patches. While this is possible, it does not give the possibility to obtain a sufficiently fine-grained refinement towards singularities.

So, we apply the following strategy; However, many alternatives are possible that fit in the theoretical framework presented in this paper.
For each marked patch, say $\Omega_k$, we first apply dyadic refinement, i.e., we construct a finer space $V_{h/2}^{(k)}$ by introducing new knots in the middle between any two non-equal knots in the knot vectors $\Xi^{(k,1)}$ and $\Xi^{(k,2)}$. Secondly, we subdivide (split) the marked patch into four sub-patches ${\Omega}_{k,i}:= G_{k,i}(\wh\Omega)$ with
\begin{equation}\label{eq:patchsplit}
\begin{aligned}
		G_{k,1}(x,y) &= G_k(\tfrac12x,\tfrac12y), &&
		G_{k,2}(x,y) = G_k(\tfrac12(x+1),\tfrac12y),\\
		  G_{k,2}(x,y) &= G_k(\tfrac12x,\tfrac12(y+1)), &&
		G_{k,4}(x,y) = G_k(\tfrac12(x+1),\tfrac12(y+1)).
\end{aligned}
\end{equation}
The associated function spaces are then $V_{h/2}^{(k,i)}:=\{ v|_{\tilde \Omega_{k,i}} : v\in V_{h/2}^{(k)}\}$, see Figure~\ref{fig:localrefine} for an example, where $\Omega_2$ is subdivided. After refinement, all resulting patches are again enumerated, like $\tilde\Omega_1,\ldots,\tilde\Omega_5$ in the figure.
By using this refinement strategy, the number of degrees of freedom (per patch) is kept almost unchanged; we have $\widehat h_{k,i} \approx \widehat h_k$ for $i\in\{1,2,3,4\}$. 

The diameters of the patches $\tilde{\Omega}_{k,i}$ are (up to multiplicative constants depending on $C_S$ from \thref{ass:map})  half the diameter of the patch $\Omega_k$. Indeed, we have
\begin{Lemma}
	Provided that the geometry representation satisfies \thref{ass:map} and the patch $\Omega_k$ is replaced by the patches $\tilde{\Omega}_{k,i}:=\tilde G_{k,i}(\wh\Omega)$ with $\tilde G_{k,i}$ as in~\eqref{eq:patchsplit} for $i\in\{1,2,3,4\}$, then the new geometry representation satisfies \thref{ass:map} with the same constant $C_G$ and $\tilde H_{k,i} = \tfrac12 H_k$.
\end{Lemma}
\begin{proof}
	Follows from simple scaling arguments and $\|\cdot\|_{L^\infty(\tilde \Omega)}
	\le \|\cdot\|_{L^\infty(\wh \Omega)}$ for $\tilde \Omega \subseteq \wh \Omega$.
\end{proof}

\begin{figure}
\centering
\begin{minipage}{0.4\textwidth}
\begin{tikzpicture}
    \node[anchor=south west,inner sep=0] (image) at (0,0) 
    {\includegraphics[width=\textwidth]{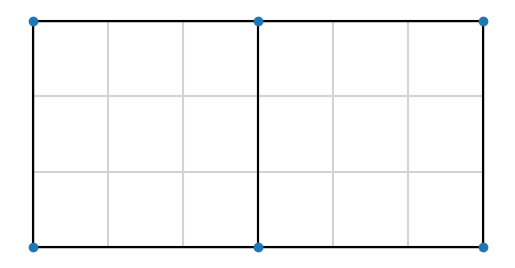}};
    \draw (1.75,1.55) node {$\Omega_1$};
    \draw (4.4,1.55) node {$\Omega_2$};
\end{tikzpicture}
\end{minipage}
\begin{minipage}{0.4\textwidth}
\begin{tikzpicture}
\node[anchor=south west,inner sep=0] (image) at (0,0) {\includegraphics[width=\textwidth]{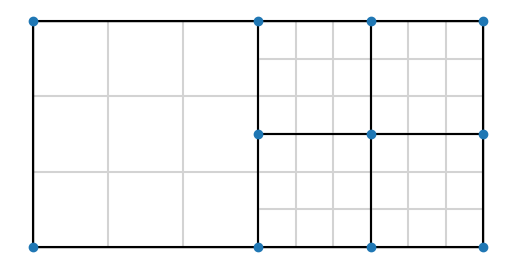}};
    \draw (1.75,1.55) node {$\Omega_1$};
    \draw (3.9,2.25) node[scale=0.75] {$\Omega^{(1,2)}_{2}$};
    \draw (5.2,2.25) node[scale=0.75] {$\Omega^{(2,2)}_{2}$};
    \draw (5.2,0.92) node[scale=0.75] {$\Omega^{(2,1)}_{2}$};
    \draw (3.9,0.92) node[scale=0.75] {$\Omega^{(1,1)}_{2}$};
    \draw (2.85,1.55) node[scale=0.75] {$\x_1$};
\end{tikzpicture}
\end{minipage}
\caption{\label{fig:localrefine}Local refinement of patch $\Omega_2$.}
\end{figure}
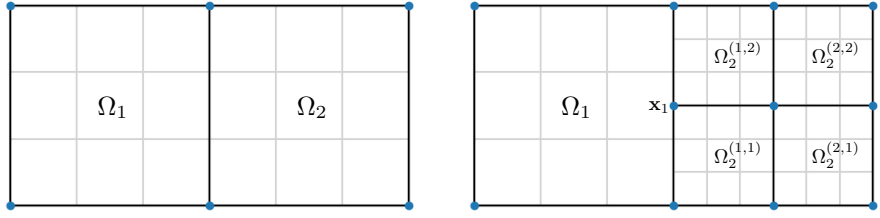

As already mentioned,  we assume that the initial configuration is matching, i.e., that for any two different patches, $\cl{\Omega_k}\cap \cl{\Omega_\ell}$ is either empty, a common corner or a common edge. By applying a splitting approach as introduced in~\eqref{eq:patchsplit}, this condition is not retained, since corners of one patch can be located on the edge of a neighboring patch; we call such corners \emph{T-junctions}. In the example depicted in Figure~\ref{fig:localrefine}, the vertex $\x_1$ would be an example for a T-junction. Our refinement strategy guarantees that the following assumptions are satisfied.
\begin{Assumption}\label{ass:matching}
	The intersection of the closures of any two different patches, i.e., $\cl{\Omega_k}\cap \cl{\Omega_\ell}$, is either (a) empty, (b) a vertex of (at least) one of the two patches, or (c) an edge of (at least) one of the two patches.
\end{Assumption}
\begin{Assumption}\label{ass:simple:t}
    No T-junction is located on $\partial \Omega$ and any two patches meeting in a T-junction share an edge.
\end{Assumption}

As mentioned in the last section, we assume~\eqref{eq:if-agrees} for the initial configuration. If only one of the patches adjacent to an edge is refined, the condition~\eqref{eq:if-agrees} is not satisfied anymore. However, the trace spaces of the neighboring patches are nested, cf. Figure~\ref{fig:localrefine}. So, the following assumption holds.
\begin{Assumption}\label{ass:nested}
For any two patches $\Omega_k$ and $\Omega_\ell$ sharing an edge, we have
\[
			V_h^{(k)}|_{\partial\Omega_k\cap\partial\Omega_\ell}
			\subseteq  V_h^{(\ell)}|_{\partial\Omega_k\cap\partial\Omega_\ell}
			\quad\mbox{or}\quad
			V_h^{(\ell)}|_{\partial\Omega_k\cap\partial\Omega_\ell}
			\subseteq  V_h^{(k)}|_{\partial\Omega_k\cap\partial\Omega_\ell}.
\]
\end{Assumption}
This configuration guarantees that the trace space is rich enough and does not degenerate the approximation quality of the discrete space. This assumption is also valid after any further refinement step. This is obvious if both patches adjacent to an edge are refined or if the patch is refined which already has a finer grid. \thref{ass:nested} is also true if the patch with the coarser grid is refined (like if one would refine $\Omega_1$ in the example depicted in Figure~\ref{fig:localrefine}) since the finer grid is the result of the application of the refinement procedure and the application of the same procedure to the coarser grid yields the same result.

For the approximation error estimates, we need the following assumption that guarantees that the edges are not too small, compared to the grid sizes of the adjacent patches. 
\begin{Assumption}\label{ass:compatibility}
    For any two patches $\Omega_k$ and $\Omega_\ell$, sharing an edge, the length of its pre-image is at least as large as $p$ times the local grid size, i.e., $|G_k^{-1}(\partial\Omega_k\cap \partial\Omega_\ell)| \ge p\, \wh h_k$.
\end{Assumption}
This condition can be ensured by guaranteeing that the local grid size disparity is not too large.

A minor drawback of this patch-wise splitting approach in comparison to alternative spline based local refinement methods is the increase of degrees of freedom, emerging from the introduction of additional (only continuous) interfaces. Since we assume to have at least $p^d$ knots per patch, where $d$ is the spatial dimension, this only moderately increases the number of degrees of freedom. If the number of knots were significantly larger this increase would even be negligible.

\section{Construction of a basis}\label{sec:4}

Although the definition \eqref{globalsp} of the global function space $V_h$ is simple, its representation in form of a basis is not straightforward, especially for the not fully matching case. We construct a basis $\Phi$ for the global space by building linear combinations of the basis functions in the bases $\Phi^{(k)}$ of the local spaces $V_h^{(k)}$. For this purpose we represent the continuity of functions in $V_h$ in terms of constraints.

A function $u_h$ that is patch-wise defined by $u_h|_{\Omega_k}=\sum_{i=1}^{n^{(k)}}u_i^{(k)} \phi_i^{(k)} \in V_h^{(k)}$ is in the global space $V_h$ if and only if it is continuous. The continuity can be expressed as constraints on the coefficient vectors $\underline u_h^{(k)} = (u_1^{(k)},\cdots,u_{n^{(k)}}^{(k)})^\top$.  We construct such constraints on an edge-to-edge basis, allowing also redundant constraints. Consider an edge $\Gamma_{k,\ell}:=\partial\Omega_k\cap \partial\Omega_\ell$. Both, the bases $\Phi^{(k)}$ and $\Phi^{(\ell)}$, can be restricted to that edge yielding trace bases $\Phi^{(k)}|_{\Gamma_{k,\ell}}:=\{\phi_i^{(k)}:\phi_i^{(k)}|_{\Gamma_{k,\ell}} \not\equiv 0\}$ and $\Phi^{(\ell)}|_{\Gamma_{k,\ell}}$, defined analogously. Using \thref{ass:nested}, we know that -- if the two trace bases do not agree anyway -- the basis functions in one of them, say $\Phi^{(k)}|_{\Gamma_{k,\ell}}$, can be represented as a linear combination of the basis functions in $\Phi^{(\ell)}|_{\Gamma_{k,\ell}}$:
\begin{equation}\label{eq:ecsetup0}
        \phi_i^{(k)} = \sum_{j=1}^{n^{(\ell)}} E_{i,j}^{(k,\ell)} \phi_j^{(\ell)}
        \quad\mbox{holds on $\Gamma_{k,\ell}$ for all $\phi_i^{(k)} \in \Phi^{(k)}|_{\Gamma_{k,\ell}}$.}
\end{equation}
The coefficients $E_{i,j}^{(k,\ell)}$ are non-negative and can be constructed by knot insertion algorithms, \cite{deBoor1972}.
Since B-splines interpolate on the boundary, the continuity of the mentioned function $u_h$ can be expressed by
\begin{equation}\label{eq:ecsetup}
		u_i^{(k)} - \sum_{j=1}^{n^{(\ell)}} E_{i,j}^{(k,\ell)} u_j^{(\ell)} = 0
        \quad
        \mbox{for all $\Omega_k$ and $\Omega_\ell$ sharing an edge and all $i$ with $\phi_i^{(k)} \in \Phi^{(k)}|_{\Gamma_{k,\ell}}$.}
\end{equation}
All of these constraints can be collected in matrices $C_k$ such that~\eqref{eq:ecsetup} is equivalent to
\[
        \sum_{k=1}^K C_k \underline u_h^{(k)} = 0.
\]

By collecting $C:=(C_1,\ldots,C_K)$ and $\underline u_h^\top = (\underline u_h^{(1)\top},\cdots,\underline u_h^{(K)\top})^\top \in \RR^{\npw}$, the condition reduces to
\[
			C \underline u_h = 0.
\]
By construction, we know that each row of $C$ has exactly one positive coefficient, which is always $1$.  Since the constraints might by redundant, the matrix might not have full rank.

As mentioned, the basis functions $\phi_i$ in the global basis $\Phi$ are linear combinations of the local basis functions:
\[
	\phi_i|_{\Omega^{(k)}} =
    \sum_{j=1}^{n^{(k)}}
    B^{(k)}_{i,j} \phi_{j}^{(k)}
\]
with coefficients $B^{(k)}_{i,j}$ to be determined. Using
\[
    B_k:= [B^{(k)}_{i,j}]_{i=1,\ldots,n}^{j=1,\ldots,n^{(k)}} \in \RR ^{ n \times n^{(k)} }
    \quad\mbox{and}\quad
    B=(B_1^\top,\cdots,B_K^{\top})^\top  \in \RR ^{ n \times  \npw },
\]
the problem of finding a basis can be rewritten as a problem on matrices. We are interested in a (full-rank) matrix $B$ whose image is the nullspace of the matrix $C$, namely
\begin{equation}\label{eq:BC}
		\nullspace  C = \image  B.
\end{equation}
Moreover, it should satisfy the desirable properties of the B-spline bases, namely the partition of unity, the non-negativity and the local support. These conditions translate to conditions on the matrix $B$:
\begin{align}
		\label{eq:partition:of:unity}
        \sum\nolimits_{j=1}^{n} B_{i,j} = 1 & \qquad\mbox{for all $i=1,\ldots,\npw$ (partition of unity),} \\
		\label{eq:non:negativity}
        B_{i,j} \ge 0 & \qquad \mbox{for all $i=1,\ldots,\npw$ and $j=1,\ldots,n$ (non-negativity).}
\end{align}
Moreover, the basis function should have a uniformly bounded support, i.e., each row of $B$ should only have a limited number of non-zero coefficients.

For simple setups, it is possible to derive a matrix $B$ satisfying these properties directly, i.e., it is possible to specify the resulting basis explicitly, this approach is hard to generalize to more involved situations, like three dimensional problems. Moreover, such direct approaches tend to be hard to implement. The following algorithm allows to construction of the matrix $B$ from the constraint matrix $C$. The matrix is a variant of Gaussian elimination, where the row is chosen in a way such that a rather small fill-in is expected.

\begin{Algorithm} $ $\label{alg:basis}
\begin{itemize}
    \item Let $B^{(0)}\in \RR^{\npw\times \npw}$ be an identity matrix.
    \item For $\nu=1,2,\ldots$ loop until $C B^{(\nu)} = 0$:
    \begin{itemize}
        \item Let
        $\mathcal F_m(C) := \{ n: |C_{m,n}|\ne 0, \; C_{m,n}C_{m,j} \le 0
        				\; \mbox{for all} \; j\ne n \}$.
        Choose $m_\nu$ and $n_\nu$ such that
		\begin{equation}\label{eq:alg:rowcond}
		n_\nu\in \mathcal F_{m_\nu}(CB^{(\nu-1)})
		\end{equation}
        and set
        \begin{equation}\label{eq:Rdef}
        		R^{(\nu)} := I - \frac1{\underline e_{m_\nu}^\top CB^{(\nu-1)}\underline e_{n_\nu}} \underline e_{n_\nu} \underline e_{m_\nu}^\top CB^{(\nu-1)}.
        \end{equation}
        \item Set $B^{(\nu)}:= B^{(\nu-1)} R^{(\nu)}$.
    \end{itemize}
    \item Return the matrix $B$ composed of the non-zero columns of $B^{(N)}$, where $N$ is the last index $\nu$ of the loop (i.e., such that $CB^{(N)}=0$).
\end{itemize}
\end{Algorithm}
The set $\mathcal{F}_m(C)$ contains the index of the sole positive coefficient of the $i$-th row (if existing) and the index of the sole negative coefficient of the $i$-th row (if existing). The set might be empty, if the $i$-th row vanishes or of there are several positive and several negative coefficients in that row. We will show later that there is at least one row $m_\nu$ such that $\mathcal F_{m_\nu}(CB^{(\nu-1)})$ is not empty (\thref{lem:sign:change}).
First, we show that the derived representation of the global basis has the desired properties.

\begin{Lemma}
    The returned matrix $B$ has full rank and its image spans the nullspace of $C$, i.e.,  \eqref{eq:BC}~holds.
\end{Lemma}
\begin{Proof}
	Recall Sylvester's matrix rank inequality, which states $\rank (M_1M_2)\ge \rank M_1+\rank M_2-n$ for any two matrices $M_1,M_2 \in \RR^{n\times n}$.
    It holds that $\rank R^{(\nu)} = \npw-1$ and we have $\rank B^{(\nu)} \ge \rank B^{(\nu-1)} -1 \ge \rank B^{(0)}-\nu = \npw-\nu$. Since $B$ consists of the non-vanishing columns of $B^{(N)}$, we know $\rank B = \rank B^{(N)} \ge \npw-N$. By construction, in each elimination step one column $B^{(\nu-1)}$ is eliminated, at least $N$ columns of $B^{(N)}$ vanish and thus
	\begin{equation*}
			\rank B = \rank B^{(N)} = \npw-N
			\qquad\mbox{and}\qquad
			B \in \RR^{\npw\times (\npw-N)},
	\end{equation*}
	which already shows that $B$ has full rank.
	
	Since $CB^{(N)}=0$ and $B$ is composed of the non-zero columns of $B^{(N)}$, we have
	\begin{equation*}
			\image B = \image B^{(N)} \subseteq \nullspace C.
	\end{equation*}
	Now, we show that these are equal by showing $\rank B^{(N)} + \rank C \ge \npw$. Using
	\[
			CB^{(\nu)}
			= \left(I-\frac{1}{\underline e_{m_\nu}^\top C B^{(\nu-1)} \underline e_{n_\nu}}
			C B^{(\nu-1)} \underline e_{n_\nu}\underline e_{m_\nu}^\top \right) C B^{(\nu-1)},
	\]
	where $m_\nu$ and $n_\nu$ are as chosen in iterate $\nu$ in the algorithm,
	we know $\nullspace (CB^{(\nu-1)})\subseteq \nullspace (CB^{(\nu)})$ and $\underline e_{n_\nu} \in \nullspace (CB^{(\nu)})$. Since~\eqref{eq:alg:rowcond} guarantees that $\underline e_{m_\nu}^\top CB^{(\nu-1)} \underline e_{n_\nu} \ne 0$, we know that $\underline e_{n_\nu} \not\in \nullspace (CB^{(\nu-1)})$. Thus, $\rank  CB^{(\nu)} \le \rank  CB^{(\nu-1)}-1 \le \rank C-\nu$.  By further using $C B^{(N)}=0$, we have $\rank C\ge N$. Since we have already shown $\rank B^{(N)} \ge \npw-N$, this shows $\rank B^{(N)} + \rank C \ge \npw$ and thus finishes the proof.
\end{Proof}

\begin{Lemma}
    The returned matrix $B$ has only non-negative coefficients, i.e., \eqref{eq:non:negativity}~holds.
\end{Lemma}
\begin{Proof}
    \eqref{eq:alg:rowcond} and the definition of $R^{(\nu)}$ guarantee that all its coefficients are non-negative. This result carries over to $B^{(N)}=R^{(1)}\cdots R^{(N)}$ and $B$ being composed of the non-vanishing columns of $B^{(N)}$.
\end{Proof}

\begin{Lemma}
    The returned matrix $B$ represents a partition of unity, i.e., \eqref{eq:partition:of:unity}~holds.
\end{Lemma}
\begin{Proof}
	Let $\underline e = (1,\ldots,1)^\top$.
	First,  we show that
	\begin{equation}\label{eq:cbnue}
			CB^{(\nu)}\underline e =0\quad \mbox{for all}\quad\nu=0,\ldots,N.
	\end{equation}
    by induction. Since the row-sum of the embedding matrices $E^{(k,\ell)}$ is $1$, we know by construction~\eqref{eq:ecsetup} that $C \underline e=0$, i.e., \eqref{eq:cbnue} for $\nu=0$. Assuming $CB^{(\nu-1)} \underline e=0$ for some $\nu$, we have
    \[
    		CB^{(\nu)}\underline e
    		= C B^{(\nu-1)}\underline e-\frac{1}{\underline e_{m_\nu}^\top C B^{(\nu-1)} \underline e_{n_\nu}}
    		C B^{(\nu-1)} \underline e_{n_\nu}\underline e_{m_\nu}^\top  C B^{(\nu-1)}\underline e=0,
    \]
    i.e., \eqref{eq:cbnue}. Using~\eqref{eq:cbnue} and~\eqref{eq:Rdef}, we know that
    $R^{(\nu)} \underline e = \underline e$. By induction, this shows $B^{(N)} \underline e = R^{(N)} \cdots R^{(1)} \underline e = \underline e$. Since $B$ consists of the non-zero columns of $B^{(N)}$, the desired result immediately follows.
\end{Proof}

Now, we give a result that allows us to show that in each step of the Algorithm, there is at least one row $m_\nu$ such that $\mathcal F_{m_\nu}(CB^{(\nu-1)})$ is not empty.
\begin{Lemma}\label{lem:eta:exists}
	For any $\mathcal M \subseteq \{1,\ldots,M\}$ with $|\mathcal M|\le \rank C$, there are $(\eta_m)_{m\in\mathcal M}$ such that
	\begin{align}\nonumber
			& \eta_m \in \mathcal F_m(C)
			&&\quad\mbox{for all}\quad
			m\in \mathcal M \\
			\label{eq:eta:unique}
			&\eta_m \ne \eta_{\tilde m}
			&& \quad\mbox{for all}\quad
			m\ne \tilde m, \; m,\tilde m\in\mathcal M.
	\end{align}
\end{Lemma}
\begin{Proof}
 	First, we observe that the construction~\eqref{eq:ecsetup} guarantees that each row $m$ of $C$ has exactly one positive entry $p_m$, i.e., we have $C_{m,p_m}>0$, and thus $p_m\in \mathcal F_m(C)$. For a constraint that is not related to corner or T-junction, this is depicted in Figure~\ref{fig:dof}~(a). For this case, we choose $\eta_m:=p_m$.

 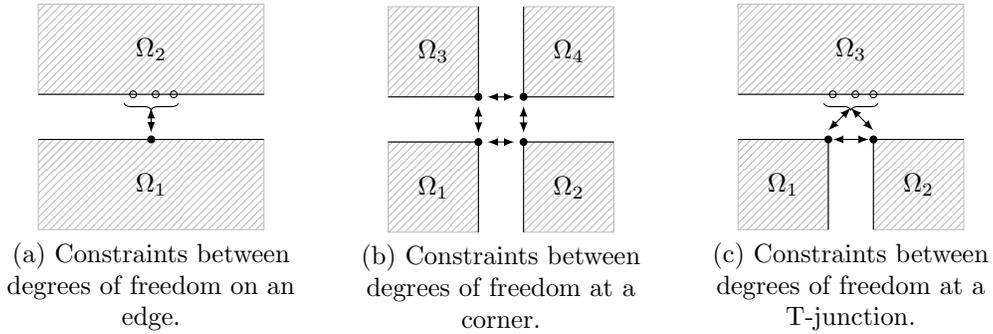
\begin{figure}[H]
    \centering
\begin{minipage}{.3\textwidth}
    \centering
	\begin{tikzpicture}[scale=.6]
        \draw[lightgray, pattern=north east lines, pattern color=lightgray] (-2.5,.5) rectangle (2.5,2.5);
        \draw[lightgray, pattern=north east lines, pattern color=lightgray] (-2.5,-.5) rectangle (2.5,-2.5);

        \draw (-2.5,.5) -- (2.5,.5);
        \draw (-2.5,-.5) -- (2.5,-.5);

        \draw [decorate,
            decoration = {brace}] (.6,0.3) --  (-.6,0.3);

        \draw[latex-latex] (0,.2) -- (0,-.35);

        \draw[black]  (-.4,.5) circle (.5ex);
        \draw[black]  (.1,.5) circle (.5ex);
        \draw[black]  (.5,.5) circle (.5ex);
        \draw[black, fill=black]  (0,-0.5) circle (.5ex);

        \draw (0,-1.5) node {$\Omega_1$};
        \draw (0,1.5) node {$\Omega_2$};
    \end{tikzpicture}

    (a) Constraints between degrees of freedom on an edge.
    \end{minipage}
    \begin{minipage}{.3\textwidth}
    \centering
    \begin{tikzpicture}[scale=.6]
        \draw[lightgray, pattern=north east lines, pattern color=lightgray] (.5,.5) rectangle (2.5,2.5);
        \draw[lightgray, pattern=north east lines, pattern color=lightgray] (-.5,.5) rectangle (-2.5,2.5);
        \draw[lightgray, pattern=north east lines, pattern color=lightgray] (.5,-.5) rectangle (2.5,-2.5);
        \draw[lightgray, pattern=north east lines, pattern color=lightgray] (-.5,-.5) rectangle (-2.5,-2.5);

        \draw (-2.5,.5) -- (-.5,.5) -- (-.5,2.5);
        \draw (-2.5,-.5) -- (-.5,-.5) -- (-.5,-2.5);
        \draw (2.5,.5) -- (.5,.5) -- (.5,2.5);
        \draw (2.5,-.5) -- (.5,-.5) -- (.5,-2.5);

        \draw[latex-latex] (-.3,-.5) -- (.3,-.5);
        \draw[latex-latex] (.3,.5) -- (-.3,.5);
        \draw[latex-latex] (-.5,.3) -- (-.5,-.3);
        \draw[latex-latex] (.5,-.3) -- (.5,.3);

        \draw[black, fill=black]  (.5,.5) circle (.5ex);
        \draw[black, fill=black]  (.5,-.5) circle (.5ex);
        \draw[black, fill=black]  (-.5,.5) circle (.5ex);
        \draw[black, fill=black]  (-.5,-.5) circle (.5ex);

        \draw (-1.5,-1.5) node {$\Omega_1$};
        \draw (1.5,-1.5) node {$\Omega_2$};
        \draw (-1.5,1.5) node {$\Omega_3$};
        \draw (1.5,1.5) node {$\Omega_4$};
    \end{tikzpicture}

    (b) Constraints between degrees of freedom at a corner.
    \end{minipage}
    \begin{minipage}{.3\textwidth}
    \centering
	\begin{tikzpicture}[scale=.6]
        \draw[lightgray, pattern=north east lines, pattern color=lightgray] (-2.5,.5) rectangle (2.5,2.5);
        \draw[lightgray, pattern=north east lines, pattern color=lightgray] (.5,-.5) rectangle (2.5,-2.5);
        \draw[lightgray, pattern=north east lines, pattern color=lightgray] (-.5,-.5) rectangle (-2.5,-2.5);

        \draw (-2.5,.5) -- (2.5,.5);
        \draw (-2.5,-.5) -- (-.5,-.5) -- (-.5,-2.5);
        \draw (2.5,-.5) -- (.5,-.5) -- (.5,-2.5);

        \draw [decorate, decoration = {brace}] (.6,0.3) --  (-.6,0.3);

        \draw[latex-latex] (0,.2) -- (-.5,-.3);
        \draw[latex-latex] (0,.2) -- (.5,-.3);

        \draw[latex-latex] (-.4,-.5) -- (.4,-.5);

        \draw[black]  (-.4,.5) circle (.5ex);
        \draw[black]  (.1,.5) circle (.5ex);
        \draw[black]  (.5,.5) circle (.5ex);
        \draw[black, fill=black]  (.5,-.5) circle (.5ex);
        \draw[black, fill=black]  (-.5,-.5) circle (.5ex);

        \draw (-1.5,-1.5) node {$\Omega_1$};
        \draw (1.5,-1.5) node {$\Omega_2$};
        \draw (0,1.5) node {$\Omega_3$};
    \end{tikzpicture}

    (c) Constraints between degrees of freedom at a T-junction.
    \end{minipage}
    \caption{Schematic representation of constraints}
    \label{fig:dof}
    \end{figure}

 	 Since the matrix $C$ is constructed on an edge-by-edge basis, it is possible that there is more than one positive coefficient in a column of $C$ if the column corresponds to a degree of freedom that is associated to more than one edge, i.e., a degree of freedom associated to the corner of a patch. These cases are depicted in Figure~\ref{fig:dof}~(b) and (c), where $\eta_m$ has to be chosen more carefully in order to guarantee~\eqref{eq:eta:unique}.

    Now, consider the case of regular corners (cf.  Figure~\ref{fig:dof}~(b)).
    
    Since B-splines interpolate at the boundary, the corresponding rows $m$ of the matrix $C$ contain exactly one positive and one negative coefficient, i.e., there is a $p_m$ and a $n_m$ such that $C_{m,n_m}<0$ and $C_{m,p_m}>0$ and $\mathcal{F}_m(C)=\{n_m,p_m\}$. For these constraints, one can choose either
    $\eta_m = n_m$ or $\eta_m = p_m$; in order to guarantee~\eqref{eq:eta:unique}, we choose a cyclic version.

    Now consider the case of T-junctions, where we have two kinds of constraints (cf.  Figure~\ref{fig:dof}~(c)). Two constraints enforce that each of the local corner values agrees with the corresponding function value on the edge, while the remaining constraint guarantee that the two corner values agree. (If more than $3$ patches would meet at a T-junction, there would be accordingly more constraints that guarantee that two corner values agree.) Again, the constraints involving two corner values correspond to a row $m$ of $C$ with exactly one positive and one negative coefficient (where $\eta_m$ could be either of them). For the other two rows we only know that there is exactly one positive coefficient (where $\eta_m$ could only be the positive one).
    
    Here, it is in general not possible to choose $\eta_m$ in a cyclic way as for the regular corners. Since one of the constraints at the T-junction is redundant (this is also valid for a regular corner, but there we did not need this argument), the set $\mathcal M$ does not contain all of these constraints. By removing one of the constraints, one can again choose $\eta_m \in \mathcal F_m(C)$ with \eqref{eq:eta:unique}.
\end{Proof}

\begin{Lemma}\label{lem:sign:change}
    In all steps of the iteration, it is possible to find some index such that \eqref{eq:alg:rowcond} holds.
\end{Lemma}
\begin{Proof}
	We make a proof by contradiction. Let $C^{(\nu)}:=CB^{(\nu)}$.
	
    Assume that after performing $\nu^*$ steps, we have $C^{(\nu^*)}\ne0$ and there is no row such that  \eqref{eq:alg:rowcond} holds.
    
    We observe that the ordering of the rows of $C$ does not have a direct effect to the algorithm; so we assume that the first $\nu^*$ rows correspond to the rows eliminated in the first $\nu^*$ steps of the algorithm, i.e., $m_1=1,m_2=2,\cdots,m_{\nu^*}=\nu^*$. Since $C^{(\nu^*)}\ne 0$, there is at least one row that does not vanish; without loss of generality, we assume that the row $\nu^*+1$ does not vanish.
    Consequently, we know for all $\nu=0,\ldots,\nu^*$ that
    (a) the first $\nu$ rows of $C^{(\nu)}$ vanish and
    (b) conversely, because of~\eqref{eq:alg:rowcond}, the row $\nu+1$ of $C^{(\nu)}$ does not vanish.
    From the combination of items~(a) and (b), we know that the first $\nu^*+1$ rows of $C$ are linear independent.

    \thref{lem:eta:exists} with $\mathcal M:=\{1,2,\ldots,\nu^*+1\}$ gives the following statement for $\nu=0$:
    \begin{equation}\label{eq:the:property}
    	\mbox{There are }
    	\eta_{\nu+1}^{(\nu)}\ne \ldots\ne \eta_{\nu^*+1}^{(\nu)}
    	\mbox{ such that  }
    		\eta_m^{(\nu)} \in \mathcal F_m(C^{(\nu)})
			\mbox{ for all }
			m=\nu,\ldots,\nu^*+1.
    \end{equation}
    If we know this statement for $\nu=\nu^*$, we know that Algorithm \ref{alg:basis} can proceed after the step $\nu^*$, which is a contradiction to the assumption that this was not the case.

   We show~\eqref{eq:the:property} by induction; we assume that it holds for some $\nu-1$. In the $\nu$-th step of the algorithm, we choose $m_\nu=\nu$ (as assumed above) and some $n_\nu \in \mathcal F_{m_\nu}(C^{(\nu-1)})$.
	Now, we have to consider two cases.

	The \emph{first case} is $n_\nu=\eta_{\nu}^{(\nu-1)}$. (We know from~\eqref{eq:the:property} that $\eta_{\nu}^{(\nu-1)} \in \mathcal F_{m_\nu}(C^{(\nu-1)})$, i.e., that it is a feasible choice.) Now consider any $i\in \{\nu+1,\ldots,\nu^*+1\}$. Without loss of generality, assume that $C^{(\nu-1)}_{i,\eta_i^{(\nu-1)}}>0$ (the case $<0$ is analogous). Using $n_\nu = \eta_{\nu}^{(\nu-1)} \ne \eta_{i}^{(\nu-1)}$, we have
	\[
				C^{(\nu)}_{i,j}
				=
				\underbrace{C^{(\nu-1)}_{i,j}}_{\displaystyle\le 0}
				- \underbrace{\frac{C^{(\nu-1)}_{\nu,j}}{C^{(\nu-1)}_{\nu,n_\nu}}}_{\displaystyle\le 0}
				\underbrace{C^{(\nu-1)}_{i,n_\nu}}_{\displaystyle\le0}
				\le 0
				\mbox{ for all }
				j\not \in \{ n_\nu, \eta_i^{(\nu-1)}\}
				\quad\mbox{and}\quad
				C^{(\nu)}_{i,n_\nu} = 0.
		\]
		This means that $j=\eta_i^{(\nu-1)}$ is the only candidate to obtain $C^{(\nu)}_{i,j}>0$. Since the $i$-th row does not vanish and the row-sum is zero, we know that $C^{(\nu)}_{i,\eta_i^{(\nu-1)}}>0$ and thus
		$\eta_i^{(\nu)} :=\eta_i^{(\nu-1)} \in \mathcal F_{\nu}(C^{(\nu-1)})$. Since this holds for all $i$, we know~\eqref{eq:the:property} for $\nu$ in the first case.

		Now, consider the \emph{second case}, i.e., $n_\nu \ne \eta_{\nu-1}^{(\nu-1)}$. Since the set $\mathcal F_\nu$ can have at most two members and using~\eqref{eq:the:property} and~\eqref{eq:alg:rowcond}, we know that
		$\mathcal F_{\nu}(C^{(\nu-1)} )=\{ n_\nu , \eta^{(\nu-1)}_{\nu} \}$. In this case,
		$C_{\nu,j}^{(\nu-1)}=0$ for $j\not\in \{n_\nu , \eta^{(\nu-1)}_{\nu} \}$. Since the row sum of $C^{(\nu-1)}$ is known to be zero, we have
		$C^{(\nu-1)}_{\nu,n_\nu} = - C^{(\nu-1)}_{\nu,\eta_{\nu}^{(\nu-1)}}$. Thus, we have
		\[
				C^{(\nu)}_{i,j} =
				\begin{cases}
				0 & \mbox{if } j = n_\nu \\
				C^{(\nu-1)}_{i,\eta_{\nu}^{(\nu-1)}} + C^{(\nu-1)}_{i,n_\nu} & \mbox{if } j = \eta_{\nu}^{(\nu-1)} \\
				C^{(\nu-1)}_{i,j} & \mbox{otherwise.} \\
				\end{cases}
		\]
		So, if $\eta_i^{(\nu-1)}\ne n_\nu$, then $\eta_i^{(\nu)}:=\eta_j^{(\nu-1)}\in\mathcal F_i(C^{(\nu)})$ with the same arguments as above. Conversely, if $\eta_i^{(\nu-1)}= n_\nu$, then $\eta_i^{(\nu)}:=\eta_{\nu}^{(\nu-1)}\in \mathcal F_i(C^{(\nu)})$. Since this holds for all $i$, we know~\eqref{eq:the:property} for $\nu$ also in the second case.

		By induction, we obtain \eqref{eq:the:property} for $\nu^*+1$. This is in contradiction to the assumption that the algorithm cannot proceed after step $\nu^*$. This finishes the proof.
\end{Proof}

\section{Approximation error estimates}\label{sec:5}

In this section, we investigate the approximation power of the multi-patch spline space $V_h$, as introduced in~\eqref{globalsp}.
Specifically, give an approximation error estimate in terms of the local mesh sizes, while also considering the patch-wise prescribed regularity of the solution. We assume $u \in H^{1+q_k}(\Omega_k)$, where $q_k \in (0,p]$ is an appropriate regularity parameter. If $q_k$ is an integer, the associated norm $\|u\|_{H^{1+q_k}(\Omega_k)}$ is the standard norm based on the $L^2$-norms of all partial derivatives of total order $\le 1+q_k$. If $q_k$ is not an integer, we assume $\|u\|_{H^{1+q_k}(\Omega_k)}$ to be as defined by real Hilbert space interpolation of  $\|u\|_{H^{1+\lfloor q_k\rfloor }(\Omega_k)}$  and $\|u\|_{H^{1+\lceil q_k\rceil }(\Omega_k)}$ via the K-method, see \cite{Adams2003}. Certainly, other ways of defining the norm lead to the same space, however only with \emph{equivalent} norms. For simplicity, we collect the local regularities to a vector $q=(q_1,\ldots,q_K)$ and write $ \mathcal H^{1+q}(\Omega) :=\{ u\in H^1(\Omega): u|_{\Omega_k} \in H^{1+q_k}(\Omega_k) \}$. Additionally we assume uniform grids for all patches.
\begin{Assumption}\label{ass:uniform}
    For each patch $\Omega_k$, there is a specific grid size $\wh h_k$ such that the corresponding knot vectors are uniform, i.e., of the form $\Xi^{(k,1)} = \Xi^{(k,2)} = (0,\cdots,0, \wh h_k, 2\wh h_k, \cdots, 1-\wh h_k, 1,\cdots,1)$.
\end{Assumption}

Within this section, we give a proof for the following two main theorems.
\begin{Theorem}\label{thrm:main-ho}
	Assume that the Assumptions~\ref{ass:map}, \ref{ass:matching}, \ref{ass:nested}, \ref{ass:compatibility}, \ref{ass:simple:t} and \ref{ass:uniform} hold. Then, there is a constant $c>0$ that only depends on the constant $C_G$ from \thref{ass:map} such that for any $q=(q_1,\ldots,q_K)$ with $q_k \in [1, p]$, the following approximation error estimate holds:
   	\[
			\inf_{u_h \in V_h}  \|u-u_h\|_{H^1(\Omega)}^2
			\le c
			\sum_{k=1}^K h_k^{2q_k} \|u\|_{H^{1+q_k}(\Omega_k)}^2
			\quad\mbox{for all}\quad
			u \in \mathcal H^{1+q}(\Omega).
	\]
\end{Theorem}

In case of reduced regularity, the estimate weakens slightly and we obtain the following result.
\begin{Theorem}\label{thrm:main-lo}
	Assume that the Assumptions~\ref{ass:map}, \ref{ass:matching}, \ref{ass:nested}, \ref{ass:compatibility}, \ref{ass:simple:t} and \ref{ass:uniform} hold. Then, for any $q=(q_1,\ldots,q_K)$ with $q_k \in (0, p]$ and any $\varepsilon \in (0,\min_kq_k)$, the following approximation error estimate holds:
   	\[
			\inf_{u_h \in V_h} \|u-u_h\|_{H^1(\Omega)}^2
			\le c_\varepsilon 
            p \max_\ell \left(1+\log \tfrac{H_\ell}{h_\ell} + \log p\right) |\VV(\Omega_\ell)|
            \sum_{k=1}^K h_k^{2(q_k-\varepsilon)} \|u\|_{H^{1+q_k}(\Omega_k)}^2,
    \]
    where $c_\varepsilon$ is a constant that only depends on $\varepsilon$ and the constant $C_G$ from \thref{ass:map} and $|\VV(\Omega_k)|$ is the number of vertices (T-junctions and corners) located on the boundary of the patch $\Omega_k$.
\end{Theorem}

Throughout this section, we write $a\lesssim b$ if there is a constant $c>0$ that only depends on the constant $C_G$ from \thref{ass:map} such that $a\le c\, b$. If $a\lesssim b$ and $b\lesssim a$, we write $a\eqsim b$.

We derive the estimate from Theorems~\ref{thrm:main-ho} and \ref{thrm:main-lo} constructively in three steps. First, we consider local projections on each of the patches and present error estimates for those. These local projections can be interpreted as a global approximation by a function that is typically discontinuous between patches. Then, we employ corrections guaranteeing continuity; this is done in two steps: first for each of the vertices, then for the edges between any two patches.

\begin{Remark}
We have introduced \thref{ass:uniform} such that we can write down the mentioned correction terms more easily. An extension to the case of non-uniform grids is, although very technical, a straight-forward extension. Since we use results from~\cite{Schneckenleitner} to prove that theorem, the estimate from \thref{thrm:main-lo} would then depend on the quasi-uniformity of the grid within each patch, i.e., on $\max_k \wh h_k/\wh h_{k,min}$.
\end{Remark}

\begin{Remark}
Note that the constants $c$ and $c_\varepsilon$ are independent of the local grid sizes $h_k$, the diameters of the patches $H_k$, the spline degree $p$, and the smoothness of the B-splines.
\end{Remark}

\subsection{Geometric setup}

In order to derive the theory, we need some more notation. We already know that the overall computational domain $\Omega$ is composed of $K$ disjointed patches $\Omega_k$, each being the image of a parameterization, i.e., $\Omega_k=G_k(\wh\Omega)$, where $\wh\Omega=(0,1)^2$. The interfaces between these patches are edges and vertices. Each patch $\Omega_k$ has four corners $\{ G_k( \wh \x ) : \wh \x \in \{0,1\}^2 \}$. We call the  points in the interior of $\Omega$ that are the corner of a patch \emph{vertex} and enumerate all vertices, i.e., we have
\[
	\{ \x_1, \ldots, \x_M \} =
	\{ G_k( \widehat{\x}^{(k)}_m ) : \widehat{\x}^{(k)}_m \in \{0,1\}^2,
         \; G_k( \widehat{\x}^{(k)}_m ) \not\in\partial\Omega \}
    \quad \mbox{with $\x_m\ne \x_n$ for $m\ne n$.}
\]
It can happen that a vertex is located on the edge of an adjacent patch, rather than on one of its corners; we call these vertices \emph{T-junctions}. Vertices that are the corner of all adjacent patches are called \emph{corner vertices}; in the example of Figure~\ref{fig:geo-setup}, $\x_1$ is a T-junction and all other vertices are a corner vertices. Analogously, we consider the edges between patches. Let
\[
	\{ \Gamma_1, \ldots, \Gamma_{J} \} 
    =
    \{ \partial \Omega_k \cap \partial\Omega_\ell
        : (k,\ell) \mbox{ arbitrarily with } |\partial \Omega_k \cap \partial\Omega_\ell| > 0 \}
    \quad \mbox{with $\Gamma_i\ne \Gamma_j$ for $i\ne j$.}
\]
be the set of distinct edges. Note that the definition does not include segments on the boundary. An edge is the segment of an interface between two patches that is enclosed between two vertices.

\begin{figure}[ht]
    \centering
    \begin{tikzpicture}
    \node[anchor=south west,inner sep=0] (image) at (0,0) {\includegraphics[width=.45\textwidth]{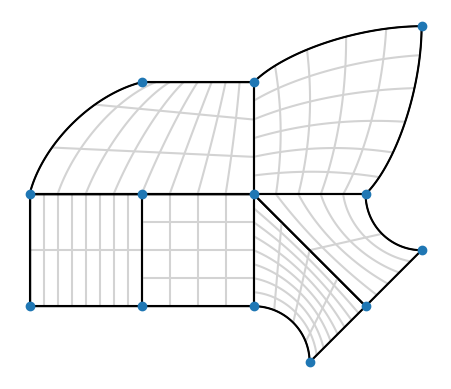}};
    \node at  (2.6,3.75) {$\Omega_1$};
    \node at  (1.4,2.1) {$\Omega_2$};
    \node at  (3.05,2.1) {$\Omega_3$};
    \node at  (4.55,1.45) {$\Omega_4$};
    \node at  (5.5,2.3) {$\Omega_5$};
    \node at  (4.95,4) {$\Omega_6$};

    \node at  (1.4,2.75) {$\Gamma_1$};
    \node at  (3.05,2.75) {$\Gamma_2$};
    \node at  (2.35,2.1) {$\Gamma_3$};
    \node at  (4.05,2.1) {$\Gamma_4$};
    \node at  (5,2.1) {$\Gamma_5$};
    \node at  (4.8,3.1) {$\Gamma_6$};
    \node at  (4.05, 3.75) {$\Gamma_7$};

    \node at  (2.25,3.1) {$\x_1$};
    \node at  (3.65,3.1) {$\x_2$};
    
    \end{tikzpicture}
    \caption{Example with $5$ patches $\Omega_k$, $7$ edges $\Gamma_i$,
    T-junction $\x_1$ and corner vertex $\x_2$.}
    \label{fig:geo-setup}
\end{figure}
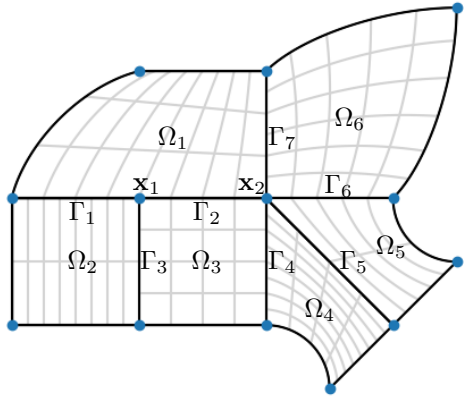

We denote by $\CC(\Omega_k)$ the indices of the vertices $\x_m$ that are located on the corners of $\Omega_k$ and by $\TT(\Omega_k)$ the indices of the vertices that are located on the remainder of $\partial\Omega_k$.
$\VV(\Omega_k):=\TT(\Omega_k) \cup \CC(\Omega_k)$ and $\EE(\Omega_k)$ denote the indices of the vertices and edges, respectively, that are located on $\partial\Omega_k$.
Analogously, $\VV(\Gamma_i)$ refers to the $2$ vertices that enclose the edge $\Gamma_i$. 
Conversely, we denote by $\PP(\x_m):=\{k:m\in \VV(\Omega_k)\}$, $\PP(\Gamma_i):=\{k:i\in \EE(\Omega_k)\}$, $\EE(\x_m):=\{i:m\in \VV(\Gamma_i)\}$ the patches and edges that are adjacent to the vertex $\x_m$ or the edge $\Gamma_i$.

Following this pattern, $\VV:=\bigcup_{k=1}^K \VV(\Omega_k)$,
$\TT:=\bigcup_{k=1}^K \TT(\Omega_k)$,
$\CC:=\VV\setminus \TT$ and
$\EE:=\bigcup_{k=1}^K \EE(\Omega_k)$
refer to all vertices, all T-junctions, all corner vertices and all edges, respectively. In the example of Figure~\ref{fig:geo-setup}, we have $\x_1\in \TT$ since it is a T-junction and $\x_1 \in \TT(\Omega_1)$ since $\x_1$ is not located on a corner of $\Omega_1$, but $\x_1 \in \CC(\Omega_2)$ and $\x_1 \not\in \TT(\Omega_2)$ since $\x_1$ is located on a corner of $\Omega_2$.

\subsection{Patch-local quasi-interpolation}

In this section, we want to recall some standard quasi-interpolation error estimates, which we use for a patch-local construction of a quasi-interpolation operator.

The error estimates are based on the $H^1$-orthogonal projector for functions in one dimension. The projector $\wh \Pi_h^{(k,\delta)}: H^1(0,1) \to \wh V_h^{(k,\delta)}$ is defined via the orthogonality property
$(\wh u-\wh \Pi_h^{(k,\delta)} \wh u, \wh v_h)_{H^1_D(0,1)}=0$ for all $\wh v_h \in \wh V_h^{(k,\delta)}$, where
\begin{equation*}
    (\wh u,\wh v)_{H^1_D(0,1)}
    :=
    \int_0^1 \wh u(x) \wh v(x) \, \mathrm dx + \wh u(0) \wh v(0).
\end{equation*}
Besides being $H^1$-orthogonal, the constant values are chosen such that $\wh \Pi_h^{(k,\delta)}\wh u(0) =  \wh u(0)$. In~\cite{Takacs2018}, it was shown that also $\wh \Pi_h^{(k,\delta)}\wh u(1) =  \wh u(1)$ holds. \cite{Sande2022}~gives the following estimates:
\begin{equation}\label{eq:t18:1d}
    \| \wh u -  \wh \Pi_h^{(k,\delta)} \wh u \|_{L^2(0,1)}
    \le
    \left(\tfrac{\wh h_k}{\pi}\right)^{1+q_k}
    | \wh u |_{H^{1+q_k}(0,1)}
    \quad\mbox{and}\quad
    | \wh u -  \wh \Pi_h^{(k,\delta)} \wh u |_{H^1(0,1)}
    \le
    \left(\tfrac{\wh h_k}{\pi}\right)^{q_k}
    | \wh u |_{H^{1+q_k}(0,1)}
\end{equation}
for $q_k \in \{0,\cdots,p\}$.

There are several possibilities to extend this to functions defined on $\wh \Omega = (0,1)^2$. If $\wh u \in H^2(\wh\Omega)$, the tensor product of the projectors $\wh \Pi_h^{(k,1)}$ and $\wh \Pi_h^{(k,2)}$ is well-defined and maps into $\wh V_h$, see \cite{Takacs2018} for details. The following lemma collects the corresponding results.
\begin{Lemma}\label{lem:t18}
    There is a projector $\wh \Pi_h^{(k)} : H^2(\wh \Omega) \to \wh V_h^{(k)}$ that satisfies the following statements:
    \begin{itemize}
        \item $\wh\Pi_h^{(k)}$ interpolates at the corners of $\wh \Omega$, i.e.,
        \begin{equation}\label{eq:t18:corner}
				(\wh\Pi_h^{(k)} \wh u)(\wh\x) = \wh u(\wh\x)
                \quad \mbox{for all} \quad \wh\x \in \{0,1\}^2
                \mbox{ and }\wh u \in H^2(\wh \Omega).
		\end{equation}
        \item Its restriction to an edge is equal to a projection on that edge; particularly, 
        \begin{equation}\label{eq:t18:edge}
            | \wh u - \wh \Pi_h^{(k)} \wh u |_{H^1(\wh \Gamma)}
            =
            \inf_{\wh v_h \in \wh V_h^{(k)}}
            | \wh u - \wh v_h |_{H^1(\wh \Gamma)},
        \end{equation}
        and satisfies the approximation error estimate
        \begin{equation}
            \label{eq:t18:approx-edge}
            \| \wh u - \wh \Pi_h^{(k)} \wh u \|_{L^2(\wh \Gamma)}
            \lesssim 
            h_k
            | \wh u - \wh \Pi_h^{(k)} \wh u |_{H^1(\wh \Gamma)}
        \end{equation}
		for all sides $\wh \Gamma \in 
            \{  \{0\}\times(0,1), \{1\}\times(0,1), 
                (0,1)\times\{0\}, (0,1)\times\{1\} \}$
        and all $\wh u \in H^2(\wh \Omega)$.
		\item It minimizes the error in the $H^1$-seminorm
        \begin{equation}\nonumber
            | \wh u - \Pi_h^{(k)} \wh u |_{H^1(\wh \Omega)}
            =
            \inf_{\wh v_h \in \wh V_h^{(k)}}
            | \wh u - \wh v_h |_{H^1(\Omega)}
		\end{equation}
        and satisfies the approximation error estimates
        \begin{align}
                \label{eq:QI_err:param:l2}
        		\| \wh u - \wh \Pi_h^{(k)} \wh u \|_{L^2(\wh\Omega)}
        		&\lesssim \wh h_k^{1+q_k}
        		| \wh u |_{H^{1+q_k}(\wh \Omega)},
                \\
                \label{eq:QI_err:param:h1}
        		| \wh u - \wh \Pi_h^{(k)} \wh u |_{H^1(\wh\Omega)}
        		&\lesssim \wh h_k^{q_k}
        		| \wh u |_{H^{1+q_k}(\wh \Omega)},
                \\
                \label{eq:QI_err:param:h2}
        		\| \partial_{\wh x,\wh y}( \wh u - \wh \Pi_h^{(k)} \wh u) \|_{L^2(\wh\Omega)}
        		&\lesssim \wh h_k^{-1+q_k}
        		| \wh u |_{H^{1+q_k}(\wh \Omega)}
        \end{align}
        for all $\wh u \in H^{1+q_k}(\wh \Omega)$.
\end{itemize}
\end{Lemma}
The statements~\eqref{eq:t18:corner} and~\eqref{eq:t18:edge} are direct consequence of  \cite[Theorem~3.4]{Takacs2018}. For the case $q_k = 1$ \cite[Theorem~3.3]{Takacs2018} gives~\eqref{eq:QI_err:param:h1}; the estimates~\eqref{eq:QI_err:param:l2} and \eqref{eq:QI_err:param:h2} follow analogously from \cite[Theorems~3.1 and 3.2]{Takacs2018}. \eqref{eq:t18:approx-edge} directly follows from~\eqref{eq:t18:1d}. Using the approximation error estimates from~\cite{Sande2022}, we can immediately extend the result to arbitrary $q_k \in \{1,\hdots,p\}$. The estimates~\eqref{eq:QI_err:param:l2} and \eqref{eq:QI_err:param:h1} can be carried over to the physical patch $\Omega_k$ using the following lemma.
\begin{Lemma}\label{lem:map}
	For all patches $\Omega_k$ and all $u\in H^r(\Omega_k)$ with $r=1,\ldots,r_k+1$, the following estimates hold:
		\[
			\|u\|_{L^2(\Omega_k)}^2
			\eqsim H_k^2 \|u\circ G_k\|_{L^2(\wh \Omega)}^2
			\quad\mbox{and}\quad
            \sum_{\ell=1}^r
			|u|_{H^\ell(\Omega_k)}^2
		    \eqsim
		    \sum_{\ell=1}^r
		    H_k^{2-2\ell}
		    |u\circ G_k|_{H^\ell(\wh \Omega)}^2.
		\]
\end{Lemma}
\begin{Proof}
    These results follow from \thref{ass:map} and standard chain and substitution rules.
\end{Proof}
We define $\Pi_h^{(k)}$ using the pull-back principle:
\[
    \Pi_h^{(k)} u := (\wh\Pi_h^{(k)} (u\circ G_k))\circ G_k^{-1}.
\]
Using this definition, \thref{lem:map}, \eqref{eq:QI_err:param:l2}, \eqref{eq:QI_err:param:h1}, $H_k \le \mbox{diam }\Omega \eqsim 1$ and $h_k =\wh h_k H_k$, we have
\begin{align}\label{eq:QI_err:l2}
            \| u - \Pi_h^{(k)} u \|_{L^2(\Omega_k)}
            \lesssim
            H_k
            \| \wh u - \wh \Pi_h^{(k)} \wh u \|_{L^2(\wh\Omega)}
            &\le
            H_k
            \left(\tfrac{\wh h_k}{\pi}\right)^{1+q_k}
            |\wh u|_{H^{1+q_k}(\wh\Omega)}
            \lesssim
            \left(\tfrac{h_k}{\pi}\right)^{1+q_k}
            \|u\|_{H^{1+q_k}(\Omega_k)},
            \\
            \label{eq:QI_err:h1}
            \| u - \Pi_h^{(k)} u \|_{H^1(\Omega_k)}
            \lesssim
            \| \wh u - \wh \Pi_h^{(k)} \wh u \|_{H^1(\wh\Omega)}
            &\le
            \left(\tfrac{\wh h_k}{\pi}\right)^{q_k}
            |\wh u|_{H^{1+q_k}(\wh\Omega)}
            \lesssim
            \left(\tfrac{h_k}{\pi}\right)^{q_k}
            \|u\|_{H^{1+q_k}(\Omega_k)},
            \\
            \label{eq:QI_err:h2}
            \wh h_k^{2} \| \partial_{\wh x, \wh y} (\wh u - \wh \Pi_h^{(k)} \wh u) \|_{L^2(\wh\Omega)}
            &\le
            \left(\tfrac{\wh h_k}{\pi}\right)^{q_k}
            |\wh u|_{H^{1+q_k}(\wh\Omega)}
            \lesssim
            \left(\tfrac{h_k}{\pi}\right)^{q_k}
            \|u\|_{H^{1+q_k}(\Omega_k)}
\end{align}
for all $u\in H^{1+q_k}(\Omega_k)$ with pull-back $\wh u:= u \circ G_k$.
If $q_k$ is not an integer, the same estimates can be derived both for $\lfloor q_k \rfloor$ and $\lceil q_k \rceil$; the desired result is than obtained by the Hilbert space interpolation theorem, see \cite[Theorem 7.23]{Adams2003}.

\begin{Remark}
The result from \cite{Takacs2018} is based on \cite{Takacs2016}, which holds for uniform grids (which we have assumed also in this paper). An extension to quasi-uniform grids is possible by using the results from~\cite{Sande2022}.
\end{Remark}

\begin{Remark}
    If there were no T-junctions and if the trace spaces on all edges would match, i.e.,
    \[
            V_h^{(k)}|_{\Gamma_i} = V_h^{(\ell)}|_{\Gamma_i}
    \]
    for all edges $\Gamma_i$ with adjacent patches $\Omega_k$ and $\Omega_\ell$
    ($\{k,\ell\} = \PP(\Gamma_i)$),
    then~\eqref{eq:t18:corner} and~\eqref{eq:t18:edge} would guarantee that a patch-wise defined function $u_h$ with
    $
            u_h|_{\Omega_k} := \Pi_k (u|_{\Omega_k})
    $
    would be continuous and thus satisfy $u_h\in V_h$. As a consequence, \thref{thrm:main-ho} would directly follow from~\eqref{eq:QI_err:l2} and~\eqref{eq:QI_err:h1}.
\end{Remark}
Since we only require \thref{ass:nested}, we have to construct a correction in order to guarantee the continuity of $u_h$, which is required to obtain $u_h \in V_h \subset H^1(\Omega)$.

\subsection{Traces for edges and vertices}

To achieve comparable quasi-interpolation error estimates within lower-dimensional manifolds (edges in this case), it is imperative to establish the subsequent trace inequalities. Using the fundamental theorem of calculus and Young's inequality, we immediately obtain
\[
        |u(x)|^2 \le \|u\|_{L^2(0,1)}^2 +  \|u\|_{L^2(0,1)}|u|_{H^1(0,1)}
        \le 2\eta^{-1} \|u\|_{L^2(0,1)}^2 + \eta |u|_{H^1(0,1)}^2
\]
for all $u\in H^1(0,1)$, all $x\in [0,1]$ and all $\eta \in (0,1]$.
By building corresponding tensor products, we immediately obtain as follows.
\begin{Lemma}\label{lem:para:trace}
    For all $\wh h\in (0,1]$, it holds that
    \begin{align*}
        \sum_{i \in \EE(\Omega_k)}
        \|\wh u\|^2_{L^2(\wh \Gamma_i^{(k)})}
        &\lesssim \wh h^{-1} \|\wh u\|_{L^2(\wh \Omega)}^2
        + \wh h |\wh u|_{H^1(\wh \Omega)}^2 
        \quad \mbox{for all}\quad \wh u \in H^1(\wh \Omega),\\
        \sum_{i \in \EE(\Omega_k)}
        |\wh u|^2_{H^1(\wh \Gamma_i^{(k)})}
        &\lesssim  
        \wh h^{-1} |\wh u|_{H^1(\wh\Omega)}^2
        + \wh h \|\partial_{\wh x,\wh y} \wh u\|_{L^2(\wh \Omega)}^2
            \quad \mbox{for all}\quad \wh u \in H^2(\wh \Omega),
    \end{align*}
    where $\wh \Gamma_i^{(k)}:=G_k^{-1}(\Gamma_i)$ is the pull-back of $\Gamma_i$.
\end{Lemma}
\begin{Proof}
    Let $\wh \Gamma$ be one of the four sides of $\wh \Omega$.
    Without loss of generality we assume $\wh \Gamma = (0,1) \times \{0\}$.

    By applying the fundamental theorem of calculus to $\wh u^2$ and using Young's inequality, we immediately obtain
    \begin{equation}\label{eq:para:trace}
        \|\wh u\|^2_{L^2(\wh\Gamma)} \lesssim \|\wh u\|_{L^2(\wh\Omega)}^2 +\|\wh u\|_{L^2(\wh \Omega)}\|\partial_{\wh y} \wh u\|_{L^2(\wh \Omega)}
        \lesssim \wh h^{-1} \|\wh u\|_{L^2(\wh \Omega)}^2
        + \wh h |\wh u|_{H^1(\wh \Omega)}^2 
        .
    \end{equation}
    By taking the sum over all four sides, we obtain a corresponding result for $\partial\wh\Omega$. The first statement of the lemma follows since
    $\bigcup_{i\in \EE(\Omega_k)} \wh \Gamma_i^{(k)} \subseteq \partial\wh\Omega$. (We obtain equality if $\partial\Omega_k$ does not contribute to the (Dirichlet) boundary $\partial\Omega$.) 
    The second estimate is obtained analogously by substituting the derivative of $\wh u$ in direction of $\wh\Gamma$ into~\eqref{eq:para:trace}. 
\end{Proof}
The relation between the norm on $\Gamma_i$ and on its pull-back $\wh\Gamma_i^{(k)}$ is provided by the following lemma.
\begin{Lemma}\label{lem:map-if}
  For all patches $\Omega_k$ with adjacent edge $\Gamma_i$ $(i\in \EE(\Omega_k))$ and pre-image $\widehat \Gamma_i^{(k)} = G_k^{-1}(\Gamma_i)$, the following estimates hold:
		\[
				\|u\|_{L^2(\Gamma_i)}^2
				\eqsim H_k \|u\circ G_k\|_{L^2(\wh \Gamma_i^{(k)})}^2
				\quad\mbox{and}\quad
				|u|_{H^1(\Gamma_i)}^2
			    \eqsim
			    H_k^{-1}
			    |u\circ G_k|_{H^1(\wh \Gamma_i^{(k)})}^2,
		\]
   for all $u\in L^2(\Gamma_i)$ or $u\in H^1(\Gamma_i)$, respectively.
\end{Lemma}
\begin{Proof}
    Analogously to \thref{lem:map}, this result follows using standard chain and substitution rules.
\end{Proof}

With similar arguments, we obtain the following trace theorem.
\begin{Lemma}\label{lem:trace:vtx}
    For all patches $\Omega_k$ with adjacent edge $\Gamma_i$ $(i\in \EE(\Omega_k))$ and pre-image $\widehat \Gamma_i^{(k)} = G_k^{-1}(\Gamma_i)$, the following estimates hold:
    \begin{align*}
        \sum_{m\in \VV(\Gamma_i) } |\wh u(\wh \x_m^{(k)})|^2
        &\lesssim
        \wh h_k^{-1} \|\wh u\|_{L^2(\wh\Gamma_i^{(k)})}^2
        + \wh h_k|\wh u|_{H^1(\wh \Gamma_i^{(k)})}^2
        \quad \mbox{for all}\quad \wh u \in H^1(\wh\Gamma_i^{(k)})
        \\
        \sum_{m\in \VV(\Omega_k) } |\wh u(\wh \x_m^{(k)})|^2
        &\lesssim
        \wh h_k^{-2} \|\wh u\|_{L^2(\wh\Omega)}^2
        + |\wh u|_{H^1(\wh \Omega)}^2
        + \wh h_k^{2}  \|\partial_{\wh x,\wh y}\wh u\|_{L^2(\wh\Omega)}^2
        \quad \mbox{for all}\quad \wh u \in H^2(\wh\Omega),
    \end{align*}
    where $\wh\x_m^{(k)}:=G_k^{-1}(\x_m)$ is the pull-back of the vertex $\x_m$.
\end{Lemma}
\begin{Proof}
    Let $m\in \VV(\Gamma_i)$. Because of \thref{ass:compatibility}, we know that $|\wh \Gamma_i^{(k)}| \ge \wh h_k$ and thus we have using the fundamental theorem of calculus and Young's inequality
    \[
        |\wh u(\widehat{\mathbf x}^{(k)}_m)|^2
        \lesssim
        \wh h_k^{-1} \|\wh u\|_{L^2(\wh \Gamma_i^{(k)})}^2
        + \wh h_k |\wh u|_{H^1(\wh \Gamma_i^{(k)})}^2.
    \]
    Since each edge is only adjacent to two vertices, this gives the first estimate. Analogously to \thref{lem:para:trace}, the second estimate follows from the first one.
\end{Proof}

\subsection{Extensions for vertices and edges}\label{sec:5:4}

As we are working with non-conforming grids, it is necessary to modify the standard projections described in earlier sections. For this purpose, we need a function $\psi_\eta\in \wh V_h^{(k,\delta)}$ such that $\psi_\eta(0)=1$, $\psi_\eta(1)=0$ and such that its norm is bounded as discussed below. Such functions are given by
\begin{equation*}
    \psi_{\eta}(t) := (\max\{1 - t/\eta,0\})^p 
\end{equation*}
if $\eta$ is a knot in the corresponding mesh. Since we assume that the grids on the parameter domain are uniform (\thref{ass:uniform}), this is guaranteed if $\eta\in (0,1]$ is an integer multiple of $\wh h_k$. Straight-forward computations yield as follows.
\begin{Lemma}\label{lem:phinorm}
    For every $\eta\in (0,1]$, we have
    $\|\psi_{\eta}\|^2_{L^2(0,1)} \eqsim p^{-1} \eta$ 
    and
    $|\psi_{\eta}|^2_{H^1(0,1)} \eqsim p \eta^{-1}$.
\end{Lemma}

From~\eqref{eq:t18:corner}, we immediately know that the local approximation agree at corner vertices. However this is not the case for T-junctions, for which we need to introduce additional corrections in order to guarantee continuity. For each T-junction $\x_m$, let $\kk(\x_m)$ be the index of the patch $\Omega_{\kk(\x_m)}$ such that
$
    m \in \TT(\Omega_{\kk(\x_m)})
$.
In the example depicted in Figure~\ref{fig:ext-t}, we have $\kk(\x_1)=1$.
Due to \thref{ass:simple:t}, there are two more patches that are adjacent to $\x_m$, which we call $\Omega_{\ell}$ and $\Omega_{\ell'}$. The edge between $\Omega_{\ell}$ and $\Omega_{\ell'}$ is called $\Gamma_i$. Let $\ell$ and $\ell'$ be such that 
\begin{equation}\label{eq:tcorr:edgcond}
    V_h^{(\ell)}\big|_{\Gamma_i}\subseteq V_h^{(\ell')}\big|_{\Gamma_i}.
\end{equation}
In the example depicted in Figure~\ref{fig:ext-t}, we have thus $\ell=3,\ell'=2$ and $i=3$.
\begin{figure}[ht]
    \centering
        \begin{tikzpicture}
    \node[anchor=south west,inner sep=0] (image) at (0,0) {\includegraphics[width=.45\textwidth]{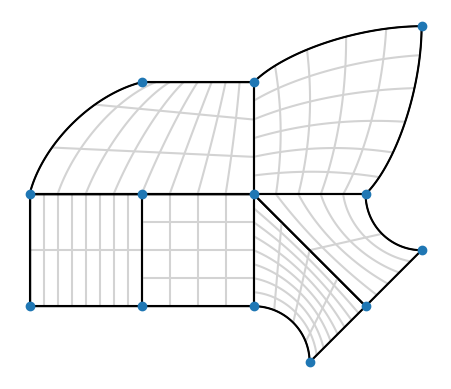}};
    
    \draw [fill=green,green, pattern = north west lines, pattern color =green] (1.32,2.11) rectangle (2.12,2.92); 
    \draw [fill=blue,blue,pattern = north west lines, pattern color =blue] (2.17,2.11) rectangle (2.99,2.92);

    \node[scale=0.75] at (1.75,2.5) {$E_{\x_1}^{(2)}$};
    \node[scale=0.75] at (2.6,2.5) {$E_{\x_1}^{(3)}$};
    
    \node at  (2.6,3.75) {$\Omega_1$};
    \node at  (1.1,1.8) {$\Omega_2$};
    \node at  (3.35,1.8) {$\Omega_3$};
    \node at  (2.35,1.85) {$\Gamma_3$};

    \node at  (2.25,3.1) {$\x_1$};
    
    \end{tikzpicture}
    \caption{\label{fig:ext-t} Extensions for the T-junction $\x_1$ are defined on $\Omega_2$ (support in green) and $\Omega_3$ (support in blue); their depth into the finer patches is adjusted by $h_{\kk(\x_m)}=h_1$, still aligned with the grids on which they are defined. The extensions agree on the adjacent edge $\Gamma_3$ and vanish on all other vertices.}
\end{figure}

The extension operator $E_{\x_m}^{(\ell)}:\RR \to V_h^{(\ell)}$ is defined by
$E_{\x_m}^{(\ell)}s = (\wh E_{\x_m}^{(\ell)} s) \circ G_\ell^{-1}$ with
\begin{equation}\label{eq:e:x:def}
        (\wh E_{\x_m}^{(\ell)}\;s)(\wh x, \wh y)
        :=
        \psi_{\eta}(\wh x)\psi_{\eta}(\wh y)
        s
        \quad \mbox{and} \quad
        \eta := \min\{ \lceil\tfrac{p h_{\kk(\x_m)}}{h_\ell} \rceil , \lfloor \tfrac{|\wh \Gamma_i^{(\ell)}|}{\wh h_\ell}\rfloor  \} \wh h_\ell ,
\end{equation}
where $\lfloor\cdot\rfloor$ and $\lceil\cdot\rceil$ denote the floor and ceiling functions.

The idea is that $\eta \sim p h_{\kk(\x_m)} H_\ell^{-1}$, but for sure not longer than the length of the pre-image of $\Gamma_i$, the edge between $\Omega_\ell$ and $\Omega_{\ell'}$.
In order to have $E_{\x_m}^{(\ell)}s \in V_h^{(\ell)}$ or, equivalently, $\wh E_{\x_m}^{(\ell)} s \in \wh V_h^{(\ell)}$, $\eta$ should correspond to a knot in the corresponding knot vector. Since we have restricted ourselves to the case that the grids on each patch are uniform (\thref{ass:uniform}), this is the case since $\eta\in(0, 1]$ is an integer multiple of $\wh h_\ell$.

For the other adjacent patch, the extension operator $E_{\x_m}^{(\ell')}:\RR \to V_h^{(\ell')}$ is defined by
$E_{\x_m}^{(\ell')}s = (\wh E_{\x_m}^{(\ell')} s) \circ G_{\ell'}^{-1}$ with
\begin{equation}\label{eq:e:x:def2}
        (\wh E_{\x_m}^{(\ell')}\;s)(\wh x, \wh y)
        :=
        \psi_{\eta}(\wh x)\psi_{\eta}(\wh y)
        s,
\end{equation}
where $\eta$ is chosen such that 
\begin{equation*}
    E_{\x_m}^{(\ell')}s\big|_{\Gamma_i}=E_{\x_m}^{(\ell)}s\big|_{\Gamma_i}.
\end{equation*}
The statement~\eqref{eq:tcorr:edgcond} and \thref{ass:uniform} guarantee that $E_{\x_m}^{(\ell')}s \in V_h^{(\ell')}$. See Figure~\ref{fig:ext-t} for a visualization. 
The extension operators can be bounded from above as follows.
\begin{Lemma}\label{lem:e:x:bound}
    Let $\Omega_k$ be a patch and the T-junction $\x_m$ be one of its corners $(m\in\TT\cap\CC(\Omega_k))$. Then, we have
    \begin{equation}\nonumber
        h_{\kk(\x_m)}^{-2} \| E_{\x_m}^{(k)}s \|^2_{L^2(\Omega_k)}
        +
       | E_{\x_m}^{(k)}s |^2_{H^1(\Omega_k)}
       \lesssim |s|^2
       \quad \mbox{for all}\quad s \in \RR.
    \end{equation}
\end{Lemma}
\begin{Proof}
    First, we estimate the size of $\eta$ from~\eqref{eq:e:x:def} or~\eqref{eq:e:x:def2}.
    
    Consider the case of~\eqref{eq:e:x:def} first, where $k=\ell$.    
    The estimate $H_\ell \wh h_\ell = h_\ell \le h_{\kk(\x_m)}$ guarantees that $\eta \le \wh h_\ell + p h_{\kk(\x_m)} H_{\ell}^{-1} \le 2p h_{\kk(\x_m)} H_{\ell}^{-1}$. \thref{ass:compatibility} guarantees $|\wh \Gamma_i^{(\ell)}|\ge p\wh h_\ell$, thus we have
    $\eta \ge \min\{ p h_{\kk(\x_m)} H_{\ell}^{-1}, \tfrac12 ||\wh \Gamma_i^{(\ell)}|| \}$. \thref{ass:nested} guarantees that $\Gamma_i$ is a full edge of $\Omega_{\ell'}$ and thus \thref{lem:map-if} and \thref{ass:compatibility} guarantee $|\wh \Gamma_i^{(\ell)}|\eqsim H_\ell^{-1}|\Gamma_i|\eqsim H_\ell^{-1} H_{\ell'}|\wh\Gamma_i^{(\ell')}|=H_\ell^{-1} H_{\ell'}\ge H_\ell^{-1} H_{\ell'} |\wh \Gamma_j^{(\ell')}|
    \eqsim H_\ell^{-1} |\Gamma_j| \gtrsim p h_{\kk(\x_m)} H_\ell^{-1}$, where $\Gamma_j$ is the edge between $\Omega_{\ell'}$ and $\Omega_{\kk(\x_m)}$. Therefore it holds for the case $k=\ell$ that 
    \begin{equation}\label{eq:etasim}
            \eta \eqsim p h_{\kk(\x_m)} H_{k}^{-1}.
    \end{equation}
    The other case~\eqref{eq:e:x:def2}, where $k=\ell'$, follows then directly 
    using \thref{lem:map-if}. 
    Using \eqref{eq:etasim} and \thref{lem:phinorm}, we obtain
    \begin{align*}
        \|\wh E_{\x_m}^{(\ell)} s\|^2_{L^2(\wh\Omega)}
        &= \|\psi^{(\ell)}_\eta\|^2_{L^2(0,1)}
        \|\psi^{(\ell)}_\eta\|^2_{L^2(0,1)}
        |s|^2
        \lesssim p^{-2} \eta^2 |s|^2
        \eqsim H_k^{-2} h_{\kk(\x_m)}^2 |s|^2,\\
        |\wh E_{\x_m}^{(\ell)} s|^2_{H^1(\wh\Omega)}
        &= |\psi^{(\ell)}_\eta|^2_{H^1(0,1)}
        \|\psi^{(\ell)}_\eta\|^2_{L^2(0,1)}
        |s|^2
        + \|\psi^{(\ell)}_\eta\|^2_{L^2(0,1)}
        |\psi^{(\ell)}_\eta|^2_{H^1(0,1)}
        |s|^2
        \lesssim |s|^2.
    \end{align*}
    We obtain the desired result using \thref{lem:map}.
\end{Proof}
We can also estimate the traces of the extension on the adjacent edges.
\begin{Lemma}\label{lem:e:x:edgebound}
    Let $\Omega_k$ be a patch, the T-junction $\x_m$ be one of its corners and $\Gamma_i$ an adjacent edge of $\Omega_k$ $(m\in \TT\cap \CC(\Omega_k)\cap \VV(\Gamma_i)$, $i\in \EE(\Omega_k))$. Then,
    \[
    h_{\kk(\x_m)}^{-1} \| E^{(k)}_{\x_m} s \|_{L^2(\Gamma_i)}^2
    + h_{\kk(\x_m)}       | E^{(k)}_{\x_m} s  |_{H^1(\Gamma_i)}^2
    \lesssim |s|^2 
    \quad\mbox{for all}\quad
    s\in \RR.
    \]
\end{Lemma}
\begin{proof}
    Using \thref{lem:map-if} and \thref{lem:phinorm} and~\eqref{eq:etasim}, we have
    \begin{align*}
        &h_{\kk(\x_m)}^{-1} \| E^{(k)}_{\x_m} s\|_{L^2(\Gamma_i)}^2
      + h_{\kk(\x_m)}       | E^{(k)}_{\x_m} s |_{H^1(\Gamma_i)}^2
      \lesssim
        h_{\kk(\x_m)}^{-1} H_k     \| \wh E^{(k)}_{\x_m} s\|_{L^2(\wh \Gamma_i^{(k)})}^2
      + h_{\kk(\x_m)}      H_k^{-1} | \wh E^{(k)}_{\x_m} s |_{H^1(\wh \Gamma_i^{(k)})}^2 
      \\& \qquad\lesssim (p^{-1} h_{\kk(\x_m)}^{-1} H_k\eta+p h_{\kk(\x_m)} H_k^{-1}\eta^{-1}) |s|^2\lesssim |s|^2,
    \end{align*}
    where $\eta$ is as in~\eqref{eq:e:x:def} or \eqref{eq:e:x:def2}, which shows the desired result.
\end{proof}
Using $\psi_\eta(t)=0$ for $t\ge \eta$, \eqref{eq:e:x:def} and \thref{ass:compatibility}, we obtain
\begin{Lemma}\label{lem:e:x:no:interference}
    Let $\Omega_k$ be a patch with corner $\x_m$, i.e., with $m\in \CC(\Omega_k)$. Then
    \[
            \big(E^{(k)}_{\x_m}s\big)(\x_n)=0
            \quad\mbox{holds for all}\quad n \in \VV(\Omega_k).
    \]
\end{Lemma}

Additionally, we need to ensure continuity across edges. For every edge $\Gamma_i$ between two patches, let $\kk(\Gamma_i)$ refer to the adjacent patch with the smaller trace space (if the trace spaces agree, $\kk(\Gamma_i)$ refers to one of the adjacent patches), i.e., such that
\begin{equation}\label{eq:edge:ext:embedded}
     V_h^{(\kk(\Gamma_i))} \big|_{\Gamma_i} \subseteq  V_h^{(\ell)} \big|_{\Gamma_i}
    \quad\mbox{for all}\quad
    \ell \in \PP(\Gamma_i).
\end{equation}
Remember that $\wh \Gamma_i^{(\ell)} = G_\ell^{-1}(\Gamma_i)$ is the pull-back of $\Gamma_i$ to the adjacent patch $\Omega_\ell$ with $\ell$ as above. Using \thref{ass:nested}, we know that $\wh \Gamma_i^{(\ell)}$ is one of the four sides of $\wh \Omega$;
without loss of generality, we assume $\wh \Gamma_i^{(\ell)} = [0,1] \times \{0\}$, i.e. the interface mapped back to the parameter domain is fixed in the $y$-component. The extension $E^{(\ell)}_{\Gamma_i} : V_h^{(\ell)}|_{\Gamma_i} \to \wh V_h^{(\ell)}$ is given by
$(E_{\Gamma_i}^{(\ell)} u) \circ G_\ell := \wh E_{\Gamma_i}^{(\ell)}(u \circ G_\ell^{-1})$ with
\begin{equation}\label{eq:e:gamma:def}
        \left(\wh E^{(\ell)}_{\Gamma_i}\; \wh u\right)(\wh x, \wh y) 
        = \wh u(\wh x,0) \psi_{\eta}(\wh y),
        \quad\mbox{with}\quad
        \eta := 
        \left\lceil \tfrac{p h_{\kk(\Gamma_i)}}{h_\ell} \right\rceil
        \wh h_\ell.
\end{equation}
Analogous to the case of vertex based extensions, we can verify using~\eqref{eq:edge:ext:embedded} and \thref{ass:compatibility} that
$\eta \in (0,1]$ and
\begin{equation}\label{eq:etasim2}
     \eta \eqsim p h_{\kk(\Gamma_i)}H_\ell^{-1}
\end{equation}
and that this operator indeed maps into $\wh V_h^{(\ell)}$, see Figure~\ref{fig:ext-e} for a visualization.
\begin{figure}[ht]
    \centering
            \begin{tikzpicture}
    \node[anchor=south west,inner sep=0] (image) at (0,0) {\includegraphics[width=.45\textwidth]{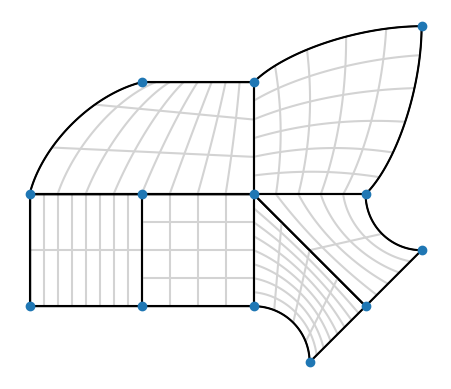}};
    \draw [fill=green,green, pattern = north west lines, pattern color =green] (0.48,2.11) rectangle (2.13,2.92); 
    \node[scale = 0.75] at (1.4,2.5) {$E_{\Gamma_1}^{(2)}$};
    \node at  (2.6,3.75) {$\Omega_1$};
    \node at  (1.4,1.8) {$\Omega_2$};
    \node at  (1.4,3.2) {$\Gamma_1$};
    \end{tikzpicture}
    \caption{\label{fig:ext-e} The edge extensions (support in green) extend into the patch with the finer grid, here $\Omega_2$; their depth is adjusted to the grid size of the patch with the coarser grid, here $\Omega_1$, still aligned with the grid on which it is defined. The extension vanishes on all corners and on all other edges.}
\end{figure}

Using this estimate and \thref{lem:map-if}, we obtain the following lemma.
\begin{Lemma}
    Let $\Omega_k$ be a patch and $\Gamma_i$ an adjacent edge $(i \in \EE(\Omega_k))$. Then, we have
    \begin{equation*}
        \|E^{(k)}_{\Gamma_i}u\|^2_{H^1(\Omega_k)}
        \lesssim h_{\kk(\Gamma_i)}^{-1}\|u\|^2_{L^2(\Gamma_i)}
        +h_{\kk(\Gamma_i)} |u|^2_{H^1(\Gamma_i)}
        \quad \mbox{for all}\quad u\in H^1_0(\Gamma_i).
    \end{equation*}
\end{Lemma}
\begin{Proof}
     Due to the tensor product structure of the extension operator, we are able to decompose the $L^2$-norm and the $H^1$-seminorm into their axial components, i.e., we obtain using the chain rule and \thref{lem:phinorm}
    \begin{align*}
        |\wh E^{(k)}_{\Gamma_i}(\wh u)|_{H^1(\wh\Omega)}^2 
        &= |\wh u|_{H^1(\wh\Gamma_i^{(k)})}^2\|\psi_{\eta}\|_{L^2(0,1)}^2
        + \|\wh u\|_{L^2(\wh\Gamma_i^{(k)})}^2|\psi_{\eta}|_{H^1(0,1)}^2
        \lesssim p^{-1} \eta |\wh u|^2_{H^1(\wh\Gamma_i^{(k)})} 
        + p \eta^{-1} \|\wh u\|^2_{L^2(\wh\Gamma_i^{(k)})}
        \\
        \|\wh E^{(k)}_{\Gamma_i}(\wh u)\|^2_{L^2(\wh\Omega)} 
        &= \|\wh u\|^2_{L^2(\wh\Gamma_i^{(k)})} \|\psi_{\eta}\|^2_{L^2(0,1)}
        \lesssim p^{-1} \eta  \|\wh u\|^2_{L^2(\wh\Gamma_i^{(k)})},
    \end{align*}
    where $\eta$ is as in~\eqref{eq:e:gamma:def}.
    Using~\eqref{eq:etasim2} and Lemmas~\ref{lem:map} and~\ref{lem:map-if}, we have
    \begin{align*}
        &\|E^{(k)}_{\Gamma_i} u\|^2_{H^1(\Omega_k)} 
        \lesssim
        H_k^{2} \|\wh E^{(k)}_{\Gamma_i}\wh u\|^2_{L^2(\wh\Omega)}
        + |\wh E^{(k)}_{\Gamma_i}\wh u|^2_{H^1(\wh\Omega)}
        \lesssim (p^{-1} H_k^2 \eta+ p\eta^{-1}) \|\wh u\|^2_{L^2(\wh \Gamma_i^{(k)})}
        + p^{-1}\eta |\wh u|^2_{H^1(\wh \Gamma_i^{(k)})} 
        \\&
        \lesssim (p^{-1} H_k \eta+ p H_k^{-1} \eta^{-1})  \|u\|^2_{L^2(\Gamma_i)}
        + p^{-1} H_k \eta  |u|^2_{H^1(\Gamma_i)}
        \lesssim (h_{\kk(\Gamma_i)} + h_{\kk(\Gamma_i)}^{-1})  \|u\|^2_{L^2(\Gamma_i)}
        +  h_{\kk(\Gamma_i)}  |u|^2_{H^1(\Gamma_i)},
    \end{align*}
    which implies the desired result since $h_{\kk(\Gamma_i)} \le H_{\kk(\Gamma_i)} \le \mbox{diam }\Omega\lesssim 1$.
\end{Proof}

Since $u\in H^1_0(\Gamma_i)$ vanishes on the two adjacent vertices and since $\eta$ is chosen such that $\psi_\eta(1)=0$, we immediately have
\begin{Lemma}\label{lem:e:gamma:no:interference}
    Let $\Omega_k$ be a patch and let $\Gamma_i$ be one of its edges, i.e., $i \in \EE(\Omega_k)$. Then, $E^{(k)}_{\Gamma_i}u$ vanishes on all other edges and on all vertices of $\Omega_k$, i.e., 
    \[
            \big(E^{(k)}_{\Gamma_i}u\big)\big|_{\Gamma_j}=0
            \quad\mbox{and}\quad
            \big(E^{(k)}_{\Gamma_i}u\big)(\x_m)=0
            \quad\mbox{for all } u \in H^1_0(\Gamma_i),\;
            j \in \EE(\Omega_k)\setminus\{i\}
            \mbox{ and }m\in \VV(\Omega_k).
    \]
\end{Lemma}

\subsection{Overall error estimate}

To obtain the overall error estimate (\thref{thrm:main-ho}), assume that $u$ to a arbitrary but fixed smooth function $\Omega \to \RR$. In the following, we construct a function $u_h \in V_h$ such that the $H^1$-error satisfies the corresponding result. First, define $u_{0,h}$ by patch-wise interpolation:
\begin{equation*}
    u_{0,h} \in L^2(\Omega)
    \quad \mbox{such that} \quad
    u_{0,h}\big|_{\Omega_k} := u_{0,h}^{(k)} := \Pi_h^{(k)} (u|_{\Omega_k}).
\end{equation*}
Using \eqref{eq:QI_err:h1}, we obtain
\begin{equation}\label{eq:u0estim}
    \sum_{k=1}^K \|u - u_{0,h} \|_{H^1(\Omega_k)}^2     
    \lesssim \sum_{k=1}^K h_k^{2q_k} \|u\|_{H^{1+q_k}(\Omega_k)}^2.
\end{equation} 

Following~\eqref{eq:t18:corner}, we know that $u_{0,h}^{(k)}(\x_m) = u(\x_m) = u_{0,h}^{(\ell)}(\x_m)$ if $\x_m$ is a common corner of $\Omega_k$ and $\Omega_\ell$ ($m\in\CC(\Omega_k)\cap\CC(\Omega_\ell)$), i.e., that $u_{0,h}$ is continuous across the corners. However, for $m\in \TT(\Omega_k)$, we have $u_{0,h}^{(k)}(\x_m) \ne u(\x_m)$ in general. In order to guarantee continuity at T-junctions as well, we define a corrected version of $u_{0,h}$ as follows:
\begin{equation}\label{eq:u1h-ho-def}
    u_{1,h} \in L^2(\Omega)
    \quad \mbox{such that} \quad
    u_{1,h}\big|_{\Omega_k} := u_{1,h}^{(k)} := u_{0,h}^{(k)} + \sum_{m\in \CC(\Omega_k) \cap \TT}
                    E_{\x_m}^{(k)} (u_{0,h}^{(\kk(\x_m))}(\x_m)-u_{0,h}^{(k)}(\x_m)).
\end{equation}
The function $u_{1,h}$ is continuous at the vertices and satisfies the same error estimate as in~\eqref{eq:u0estim}.
\begin{Lemma}
    For any two patches $\Omega_k$ and $\Omega_\ell$ with common vertex $\x_m$ $(m\in \VV(\Omega_k)\cap \VV(\Omega_\ell))$, we have $u_{1,h}^{(k)}(\x_m)=u_{1,h}^{(\ell)}(\x_m)$.
\end{Lemma}
\begin{proof}
    We have $u_{1,h}^{(k)}(\x_m)=u_{0,h}^{(\kk(\x_m))}(\x_m)$ for all $k \in\PP(\x_m)$, due to the following observations. 

    If $m\in \CC$, using \thref{lem:e:x:no:interference} and~\eqref{eq:t18:corner} we have that $u_{1,h}^{(k)}(\x_m)=u_{0,h}^{(k)}(\x_m)=u(\x_m)$ for all $k\in \PP(\x_m)$, including $\kk(\x_m)$.
    Now, consider the case $m \in \TT$.
    If $k$ is such that $m\in \CC(\Omega_k)$, we have using \thref{lem:e:x:no:interference} and the effects of extension that $u_{1,h}^{(k)}(\x_m) = u_{0,h}^{(\kk(\x_m))}(\x_m)$.
    If $k$ is such that $m\in \TT(\Omega_k)$, using \thref{lem:e:x:no:interference} and $k=\kk(\x_m)$, we have that $u_{1,h}^{(k)}(\x_m) = u_{0,h}^{(\kk(\x_m))}(\x_m)$. The last two statements guarantee continuity at $\x_m$ if it is a T-junction.
\end{proof}
\begin{Lemma}\label{lem:u1estim}
    Let $u_{1,h}$ be defined as in~\eqref{eq:u1h-ho-def}. Then, the following estimate holds:   
    \[
        \sum_{k=1}^K \| u - u_{1,h} \|_{H^1(\Omega_k)}^2    
        \lesssim \sum_{k=1}^K h_k^{2q_k} \|u\|_{H^{1+q_k}(\Omega_k)}^2.
    \]
\end{Lemma} 
\begin{proof}
    Using~\eqref{eq:u1h-ho-def}, the triangle inequality, $|\CC(\Omega_k)\cap \TT|\le 4$,
    \thref{lem:e:x:bound} and \eqref{eq:t18:corner}, we have
    \begin{align*}
        \| u - u_{1,h} \|_{H^1(\Omega_k)}^2
        & \lesssim
        \| u - u_{0,h} \|_{H^1(\Omega_k)}^2
        +
        \sum_{m\in \CC(\Omega_k)\cap \TT}
        \| E_{\x_m}^{(k)} (u_{0,h}^{(\kk(\x_m))}(\x_m) - u(\x_m) ) \|_{H^1(\Omega_k)}^2 \\
        & \lesssim
        \| u - u_{0,h} \|_{H^1(\Omega_k)}^2
        +
        \sum_{m\in \CC(\Omega_k)\cap \TT}
        |  u_{0,h}^{(\kk(\x_m))}(\x_m) - u(\x_m)   |^2.
    \end{align*}
    By rearranging the sums and using \thref{ass:simple:t} we obtain
    \begin{align*}
        \sum_{k=1}^K
        \| u - u_{1,h} \|_{H^1(\Omega_k)}^2
        & \lesssim
        \sum_{k=1}^K \| u - u_{0,h} \|_{H^1(\Omega_k)}^2
        +
        \sum_{k=1}^K
        \sum_{m\in \TT(\Omega_k)}
        |  u_{0,h}^{(k)}(\x_m) - u(\x_m)   |^2.
    \end{align*}
    Using~\eqref{eq:u1h-ho-def} for the first term and
    \thref{lem:trace:vtx}, \eqref{eq:QI_err:h1} and \eqref{eq:QI_err:h2}
    for the second term, we obtain the desired result.
\end{proof}

In order to guarantee continuity across the edges as well, we define
\begin{equation}\label{eq:u2h-ho-def}
    u_{2,h} \in L^2(\Omega)
    \quad \mbox{such that} \quad
    u_{2,h}\big|_{\Omega_k}:=u_{2,h}^{(k)} := u_{1,h}^{(k)} + \sum_{i\in \EE(\Omega_k)}
                    E_{\Gamma_i}^{(k)} (u_{1,h}^{(\kk(\Gamma_i))}-u_{1,h}^{(k)}).
\end{equation}
We obtain the desired continuity statement, as well as the desired error estimate.
\begin{Lemma}\label{lem:u2h-ho-cont}
    We have $u_{2,h}\in C^0(\Omega)$.
\end{Lemma}
\begin{proof}
Since the edge corrections vanish on all corners (\thref{lem:e:gamma:no:interference}), we have $u_{2,h}^{(k)}(\x_m)=u_{1,h}^{(k)}(\x_m)$, so continuity across vertices is retained.
Using \thref{lem:e:gamma:no:interference}, we immediately obtain $u_{2,h}^{(k)}\big|_{\Gamma_i} = u_{1,h}^{(\kk(\Gamma_i))}\big|_{\Gamma_i}$ for all $i\in \EE(\Omega_k)$ and all patches $\Omega_k$, i.e., continuity across edges. This shows continuity of the overall function.
\end{proof}

\begin{Lemma}\label{lem:u2estim}
    Let $u_{2,h}$ be defined as in~\eqref{eq:u2h-ho-def}. Then, the following estimate holds:
    \[
        \sum_{k=1}^K \| u - u_{2,h} \|_{H^1(\Omega_k)}^2 \lesssim \sum_{k=1}^K h_k^{2q_k} \|u\|_{H^{1+q_k}(\Omega_k)}^2.
    \]
\end{Lemma} 
\begin{proof}
    Using~\eqref{eq:u2h-ho-def}, the triangle inequality, the fact that there are at most $4$ extensions active on each patch (there are only $4$ indices $i \in \EE(\Omega_k)$ such that $k\ne \kk(\Gamma_i)$) and $h_{\kk(\Gamma_i)} \le H_{\kk(\Gamma_i)} \le \mbox{diam }\Omega \lesssim 1$, 
    we have
    \begin{align*}
        &\| u - u_{2,h} \|_{H^1(\Omega_k)}^2
        \lesssim
        \| u - u_{1,h} \|_{H^1(\Omega_k)}^2
        +
        \sum_{i\in \mathbb E(\Omega_k)}
        \| E_{\Gamma_i}^{(k)} (u_{1,h}^{(\kk(\Gamma_i))} - u_{1,h}^{(k)} ) \|_{H^1(\Omega_k)}^2 \\
        & \qquad \lesssim
        \| u - u_{1,h} \|_{H^1(\Omega_k)}^2
        +
        \sum_{i\in \mathbb E(\Omega_k)}
        \hnorm u_{1,h}^{(\kk(\Gamma_i))}- u_{1,h}^{(k)} \hnorm_{\Gamma_i}^2, 
    \end{align*}
    where we define $\hnorm w \hnorm_{\Gamma_i}^2 := h_{\kk(\Gamma_i)} | w |_{H^1(\Gamma_i)}^2 + h_{\kk(\Gamma_i)}^{-1} \| w \|_{L^2(\Gamma_i)}^2$ for convenience.
    In order to estimate the jumps $u_{1,h}^{(\kk(\Gamma_i))}- u_{1,h}^{(k)}$, we use \eqref{eq:u1h-ho-def}. We observe that vertex corrections are only performed for vertices $\x_m$ that are T-junctions. If $\x_m$ is a corner both of $\Omega_k$ and $\Omega_{\kk(\Gamma_i)}$ ($m\in \CC(\Omega_k)\cap \CC(\Omega_{\kk(\Gamma_i)})$), then $E_{\x_m}^{(k)}$ and 
    $E_{\x_m}^{(\kk(\Gamma_i))}$ agree on $\Gamma_i$. So, there can only be a contribution if $m\in \TT(\Omega_k)$ or $m\in \TT(\Omega_{\kk(\Gamma_i)})$. Since $\Omega_{\kk(\Gamma_i)}$ has the smaller trace space, only the latter case can occur. Thus, we have
    \begin{align*}
        \hnorm u_{1,h}^{(\kk(\Gamma_i))}- u_{1,h}^{(k)} \hnorm_{\Gamma_i}^2
        & \lesssim
        \hnorm u_{0,h}^{(\kk(\Gamma_i))}- u_{0,h}^{(k)} \hnorm_{\Gamma_i}^2
        +
        \sum_{m\in \VV(\Gamma_i) \cap \TT(\Omega_{\kk(\Gamma_i)})}
        \hnorm  E_{\x_m}^{(k)} ( u_{0,h}^{(\kk(\Gamma_i))}(\x_m) - u_{0,h}^{(k)}(\x_m)) \hnorm_{\Gamma_i}^2.
    \end{align*}
    Here, the first summand is estimated using~\eqref{eq:t18:edge} and~\eqref{eq:t18:approx-edge}. For estimating the second summand, we use $h_{\kk(\x_m)}=h_{\kk(\Gamma_i)}$ and \thref{lem:e:x:edgebound}. So, we obtain
    \begin{align*}
        \hnorm u_{1,h}^{(\kk(\Gamma_i))}- u_{1,h}^{(k)} \hnorm_{\Gamma_i}^2
        & \lesssim
        \hnorm u_{0,h}^{(\kk(\Gamma_i))}- u \hnorm_{\Gamma_i}^2
        +
        \sum_{m\in \VV(\Gamma_i) \cap \TT(\Omega_{\kk(\Gamma_i)})}
        |  u_{0,h}^{(\kk(\Gamma_i))}(\x_m) - u(\x_m)|^2.
    \end{align*}
    By rearranging the sums and using $h_k\le h_{\kk(\Gamma_i)}$ for all $i\in \EE(\Omega_k)$, we have
    \begin{align*}
        \sum_{k=1}^{K} \| u - u_{2,h} \|_{H^1(\Omega_k)}^2
        &\lesssim
        \sum_{k=1}^{K}
        \| u - u_{1,h} \|_{H^1(\Omega_k)}^2
        +
        \sum_{k=1}^{K}
        \sum_{i\in \mathbb E(\Omega_k)}
        (h_k^{-1} \| u - u_{0,h}^{(k)} \|_{L^2(\Gamma_i)}^2
        +h_k| u - u_{0,h}^{(k)} |_{H^1(\Gamma_i)}^2)
        \\&\qquad+
        \sum_{k=1}^{K}
        \sum_{m\in \TT(\Omega_k)}
        |  u_{0,h}^{(k)}(\x_m) - u(\x_m)|^2.
    \end{align*}
    Here, the first summand is estimated with \thref{lem:u1estim} and the other summands are estimated using \thref{lem:para:trace} and \thref{lem:trace:vtx}, respectively. Combined with~\eqref{eq:QI_err:l2}, \eqref{eq:QI_err:h1} and~\eqref{eq:QI_err:h2} we obtain the desired result.
\end{proof}

\begin{proof}[Proof of \thref{thrm:main-ho}]
    Using~\thref{lem:u2h-ho-cont} and $u_{2,h}^{(k)}\in V_h^{(k)}$ for all $k$, we have $u_{2,h}\in V_h$. Thus, the desired statement follows for smooth $u$ directly from \thref{lem:u2estim}. For $u\in \mathcal H^{1+q}(\Omega)$, the result follows using standard density arguments.
\end{proof}

\subsection{Overall error estimate for the low-regularity case}

In this section, we show \thref{thrm:main-lo}. We again assume that $u$ is an arbitrary but fixed smooth function first. Since we do not know that the function to be approximated is in $H^2(\Omega)$, we cannot use the projector $\Pi_h^{(k)}$, since that projector is not stable in $H^{1+q}(\Omega)$ for $q<1$. Instead, we define $u_{0,h}$ by patch-wise interpolation using a standard $H^1(\Omega)$-orthogonal projector. So, let $\wh \pi_h^{(k)}:H^1(\wh \Omega) \to \wh V_h^{(k)}$ be the projector that is orthogonal with respect to the $H^1(\wh \Omega)$-scalar product. Again, define $\pi_h^{(k)}$ via the pull-back principle, i.e., $\pi_h^{(k)} u = (\wh \pi_h^{(k)} (u\circ G_k))\circ G_k^{-1}$ and $u_{0,h}$ as follows:
\begin{equation*}
    u_{0,h} \in L^2(\Omega)
    \quad \mbox{such that} \quad
    u_{0,h}\big|_{\Omega_k} := u_{0,h}^{(k)} := \pi_h^{(k)} (u|_{\Omega_k}).
\end{equation*}
Using the results from \cite{Sande2022} and \thref{lem:map}, we know that
\begin{equation}\label{eq:sande1}
            \| u - u_{0,h} \|_{L^2(\Omega_k)}
            \lesssim
            h_k^{1+q_k}
            |u|_{H^{1+q_k}(\Omega_k)}
            \quad\mbox{and}\quad
            | u - u_{0,h} |_{H^r(\Omega_k)}
            \lesssim
            h_k^{q_k-r}
            |u|_{H^{1+q_k}(\Omega_k)}
\end{equation}
for all $q_k \in \{0,1,\ldots,p\}$ and $r\in\{1,2\}$ with $r\le q_k$. Again, the Hilbert space interpolation theorem \cite[Theorem 7.23]{Adams2003} provides 
\begin{equation}\label{eq:sandeinterpol}
            \| u - u_{0,h} \|_{H^{1+\varepsilon}(\Omega_k)}
            \lesssim
            h_k^{q_k-\varepsilon}
            |u|_{H^{1+q_k}(\Omega_k)}
\end{equation}
for all $q_k \in (0,p]$ and $\varepsilon\in(0,\min_k q_k)$. From standard Sobolev space embedding theorems \cite[Theorem 4.12]{Adams2003}, we know that for all $\varepsilon > 0$ there is a constant $C_\varepsilon \ge 1$ (depending on $\varepsilon$ and $q_k$ only) such that
\[
	\|\wh u\|_{C(\wh\Omega)}
 := \sup_{x\in\wh \Omega}|\wh u(x)| 
 \le C_\varepsilon \|\wh u\|_{H^{1+\varepsilon}(\wh \Omega)}
  \quad\mbox{for all $\wh u \in H^{1+\varepsilon}(\wh \Omega)$.}
\]
Consequently, using \eqref{eq:sandeinterpol} and \thref{lem:map}, we also know
\begin{equation}\label{eq:sande:corn}
    \| u - u_{0,h} \|_{C(\Omega_k)}
    \lesssim
		C_\varepsilon
    h_k^{q_k-\varepsilon}
    |u|_{H^{1+q_k}(\Omega_k)}
    \quad\mbox{for all }\varepsilon\in(0,q_k).
\end{equation}
Since $u_{0,h}$ does not interpolate $u$, we need corrections not only for the T-junctions, but also for the corner vertices. For each vertex $\x_m$ (corner or T-junction), the extension $E_{\x_m,lr}^{(k)}$ is defined via
\begin{equation*}
        E_{\x_m,lr}^{(k)}s = (\wh E_{\x_m,lr}^{(k)} s) \circ G_k^{-1},
        \quad\mbox{where}\quad
        (\wh E_{\x_m,lr}^{(k)}\;s)(\wh x, \wh y)
        :=
        \psi_{p \wh h_k}(\wh x)\psi_{p \wh h_k}(\wh y)
        s.
\end{equation*}
As opposed to the case of high regularity, the extension operators do not necessarily match at the edges. Again, using \thref{ass:compatibility}, we obtain that $(E_{\x_m,lr}^{(k)}s)(\x_n)=0$ for $m \ne n$.

We then set
\begin{equation}\label{eq:u1h-lo-def}
    u_{1,h} \in L^2(\Omega)
    \quad \mbox{such that} \quad
    u_{1,h}^{(k)} := u_{0,h}^{(k)} + \sum_{m\in \CC(\Omega_k)}
                    E_{\x_m,lr}^{(k)} (u_{0,h}^{(\kk(\x_m))}(\x_m)-u_{0,h}^{(k)}(\x_m)),
\end{equation}
where $\kk(\x_m)$ is defined as in Section~\ref{sec:5:4} if $m\in \TT$. If $m\in \CC$, then $\kk(\x_m)\in \PP(\x_m)$ can be chosen arbitrarily.
Using the same arguments as in the last section, we observe that $u_{1,h}$ is continuous across all vertices.
We obtain the following error estimate.

\begin{Lemma}\label{lem:u1estim:lr}
    Let $u_{1,h}$ be defined as in~\eqref{eq:u1h-lo-def}. Then, the following estimate holds:   
    \[
        \sum_{k=1}^K 
        \left(\| u - u_{1,h} \|_{H^1(\Omega_k)}^2 
        + \| u - u_{1,h} \|_{C(\Omega_k)}^2 \right)
        \lesssim 
        \sum_{k=1}^K |\VV(\Omega_k)| C_\varepsilon h_k^{2(q_k-\varepsilon)}
        \|u\|_{H^{1+q_k}(\Omega_k)}^2
    \]
		for any $\varepsilon\in(0,\min_kq_k)$.
\end{Lemma} 
\begin{proof}
    Using~\eqref{eq:u1h-lo-def} and Lemmas~\ref{lem:map} and \ref{lem:phinorm}, we obtain analogously to the proof of \thref{lem:e:x:bound}:
    \begin{align*}
        \| u - u_{1,h} \|_{H^1(\Omega_k)}^2
        & \lesssim
        \| u - u_{0,h} \|_{H^1(\Omega_k)}^2
        +
        \sum_{m\in \VV(\Omega_k)}
        \| E_{\x_m,lr}^{(k)} (u_{0,h}^{(\kk(\x_m))}(\x_m) - u_{0,h}^{(k)}(\x_m) ) \|^2_{H^1(\Omega_k)} \\
        & \lesssim
        \| u - u_{0,h} \|_{H^1(\Omega_k)}^2
        +
        \sum_{m\in \VV(\Omega_k) }
        |  u_{0,h}^{(\kk(\x_m))}(\x_m) - u_{0,h}^{(k)}(\x_m)   |^2.
    \end{align*}
    Since $\| E_{\x_m,lr}^{(k)} s \|_{C(\Omega_k)} = |s|$, we also have
    \begin{align*}
        \| u - u_{1,h} \|_{C(\Omega_k)}^2
        & \lesssim
        \| u - u_{0,h} \|_{C(\Omega_k)}^2
        +
        \sum_{m\in \VV(\Omega_k) }
        |  u_{0,h}^{(\kk(\x_m))}(\x_m) - u_{0,h}^{(k)}(\x_m)   |^2.
    \end{align*}
		By rearranging the sums and using the triangle inequality 
		$|  u_{0,h}^{(\kk(\x_m))}(\x_m) - u_{0,h}^{(k)}(\x_m)   |^2
		\lesssim 
		|  u(\x_m) - u_{0,h}^{(\kk(\x_m))}(\x_m) |^2
		+
		|  u(\x_m) - u_{0,h}^{(k)}(\x_m)   |^2$, we obtain
    \begin{align*}
        \sum_{k=1}^K \left(\| u - u_{1,h} \|_{H^1(\Omega_k)}^2 + \| u - u_{1,h} \|_{C(\Omega_k)}^2\right)
        &\lesssim
        \sum_{k=1}^K \left(\| u - u_{0,h} \|_{H^1(\Omega_k)}^2 + \| u - u_{0,h} \|_{C(\Omega_k)}^2\right)
        \\&\qquad+
        \sum_{m\in \VV(\Omega_k) }
				|  u(\x_m) - u_{0,h}^{(k)}(\x_m)   |^2.
    \end{align*}
    \eqref{eq:sande1} and \eqref{eq:sande:corn} give the desired result. 
\end{proof}

To guarantee continuity across edges, we also need edge corrections. Here, we base them on the discrete harmonic extensions, i.e., $\wh E_{\Gamma_i,lr}^{(k)}$ is a linear operator that maps from the trace space $\wh V_h^{(k)}|_{\wh \Gamma_i^{(k)}}$ to $\wh V_h^{(k)}$ that satisfies $|\wh E_{\Gamma_i,lr}^{(k)} \wh u|_{\wh \Gamma_i^{(k)}} = \wh u$ and $|\wh E_{\Gamma_i,lr}^{(k)} \wh u|_{\partial\wh\Omega\setminus\wh \Gamma_i^{(k)}} = 0$ such that it minimizes its $H^1(\wh\Omega)$-seminorm, i.e.,
$|\wh E_{\Gamma_i,lr}^{(k)} \wh u|_{H^1(\wh \Omega)}$. Due to \cite[Theorem~4.2,Lemma~4.15]{Schneckenleitner}, the following estimate holds:
\[
		|\wh E_{\Gamma_i,lr}^{(k)} \wh u|_{H^1(\wh \Omega)}
		\lesssim 
		p |\wh u|_{H^{1/2}(\partial\wh\Omega)}
		\lesssim 
		p |\wh u|_{H^{1/2}(\wh \Gamma_i^{(k)})}
		+ p (1+\log \wh h_k^{-1} + \log p) \|\wh u\|_{C(\wh \Gamma_i^{(k)})}
		\;\mbox{for all}\; \wh u \in \wh V_h^{(k)}|_{\wh \Gamma_i^{(k)}}.
\]
The Poincaré inequality and \thref{lem:map} yield for $E_{\Gamma_i,lr}^{(k)}s = (\wh E_{\Gamma_i,lr}^{(k)} s) \circ G_k^{-1}$ that
\begin{equation}\label{eq:harmon}
		\|E_{\Gamma_i,lr}^{(k)} u\|_{H^1(\Omega_k)}
		\lesssim 
		p \|u\|_{H^{1/2}(\Gamma_i)}
		+ p (1+\log \tfrac{H_k}{h_k} + \log p) \|u\|_{C( \Gamma_i)}
		\;\mbox{for all}\; u \in V_h^{(k)}|_{\Gamma_i}.
\end{equation}
By construction, we have $\big(E^{(k)}_{\Gamma_i,lr}u\big)\big|_{\Gamma_j}=0$
and
$\big(E^{(k)}_{\Gamma_i,lr}u\big)(\x_m)=0$
for all $u \in V_h^{(k)}|_{\Gamma_i}\cap H^1_0(\Gamma_i)$, all $j \in \EE(\Omega_k)\setminus\{i\}$ and $m\in \VV(\Omega_k)$.
We define
\begin{equation}\label{eq:u2h-lo-def}
    u_{2,h} \in L^2(\Omega)
    \quad \mbox{such that} \quad
    u_{2,h}^{(k)} := u_{1,h}^{(k)} + \sum_{i\in \EE(\Omega_k)}
                    E_{\Gamma_i,lr}^{(k)} (u_{1,h}^{(\kk(\Gamma_i))}-u_{1,h}^{(k)}).
\end{equation}
Analogously to the preceding section, we obtain that $u_{2,h}$ is continuous across all edges and that the continuity across the vertices is retained. The following lemma provides an error estimate.

\begin{Lemma}\label{lem:u2estim:lr}
    Let $u_{2,h}$ be defined as in~\eqref{eq:u2h-lo-def}. Then, the following estimate holds:   
    \[
            \sum_{k=1}^K \| u - u_{2,h} \|_{H^1(\Omega_k)}^2 
            \lesssim C_\varepsilon  p
                \max_\ell (1+\log \tfrac{H_\ell}{h_\ell} + \log p) |\VV(\Omega_\ell)|
        		\sum_{k=1}^K   h_k^{2(q_k-\varepsilon)} \|u\|_{H^{1+q_k}(\Omega_k)}^2.
    \]
\end{Lemma} 
\begin{proof}
    Using~\eqref{eq:u2h-lo-def} and~\eqref{eq:harmon}, we obtain
    \begin{align*}
        \| u - u_{2,h} \|_{H^1(\Omega_k)}^2
         \lesssim
        \| u - u_{1,h} \|_{H^1(\Omega_k)}^2
        &+
        p
        \sum_{i\in \EE(\Omega_k)}
        \hnorm u_{1,h}^{(\kk(\Gamma_i))} - u_{1,h}^{(k)} \hnorm_{\Gamma_i}^2,
    \end{align*}
    where $ \hnorm w \hnorm_{\Gamma_i}^2:=\|w\|_{H^{1/2}(\Gamma_i)}^2 + \Lambda \| w \|_{C(\Gamma_i)}^2$ and $\Lambda:=\max_\ell (1+\log \tfrac{H_\ell}{h_\ell} + \log p)$.	Using the triangle inequality ($\hnorm u_{1,h}^{(\kk(\Gamma_i))} - u_{1,h}^{(k)} \hnorm_{\Gamma_i} \le \hnorm u- u_{1,h}^{(\kk(\Gamma_i))}  \hnorm_{\Gamma_i} + \hnorm u - u_{1,h}^{(k)}  \hnorm_{\Gamma_i} $), by rearranging the local contributions and using $\sum_i \|w\|_{H^{1/2}(\Gamma_i)}\le \|w\|_{H^{1/2}(\partial\Omega_k)}$, we obtain
    \begin{align*}
    		\sum_{k=1}^K
        \| u - u_{2,h} \|_{H^1(\Omega_k)}^2
         \lesssim
    		\sum_{k=1}^K
        \| u - u_{1,h} \|_{H^1(\Omega_k)}^2
        &+
    		\sum_{k=1}^K p
    		(
        \| u - u_{1,h}^{(k)} \|_{H^{1/2}(\partial\Omega_k)}^2 
        +
        \Lambda
        \| u - u_{1,h}^{(k)} \|_{C(\Omega_k)}^2
        ).
    \end{align*}
   	Using a standard trace theorem \cite[Lemma 7.40]{Adams2003}, we have
    \begin{align*}
    		\sum_{k=1}^K
        \| u - u_{2,h} \|_{H^1(\Omega_k)}^2
         \lesssim
				p \Lambda 
    		\sum_{k=1}^K
        (\| u - u_{1,h} \|_{H^1(\Omega_k)}^2
        +\| u - u_{1,h}^{(k)} \|_{C(\Omega_k)}^2
        ).
    \end{align*}
   	\thref{lem:u1estim:lr} gives the desired result.
\end{proof}

\begin{proof}[Proof of \thref{thrm:main-lo}]
    Since $u_{2,h}\in C^0(\Omega)$ and $u_{2,h}^{(k)}\in V_h^{(k)}$ for all $k$, we have $u_{2,h}\in V_h$. Thus, the desired statement follows directly from \thref{lem:u2estim:lr} for $u$ smooth. For $u\in \mathcal H^{1+q}(\Omega)$, the result follows using standard density arguments.
\end{proof}

\section{Numerical experiments}\label{sec:6}

\subsection{L-shape domain}

As a first test case, we apply our approach for adaptive IgA to an L-shaped domain $\Omega = (-1,1)^2 \setminus ([0,1)\times[-1,0))$. Since the domain is non-convex, we cannot expect full $H^2$-regularity of the solution, even if $f \in L^2(\Omega)$ and $g \in H^{3/2}(\partial \Omega)$. 

The singularity in the solution is however local and known to be at the reentrant corner. By using adaptive refinement, we are able to reconstruct the expected order of convergence one would expect for the used spline space.

First, we prescribe a solution $u(x,y) = u(r \cos \varphi, r \sin \varphi) = r^{2/3} \sin(\tfrac23 \varphi)$, where $r$ and $\phi$ are radial coordinates and consider the following Poisson problem:
\begin{alignat}{2}
    -\Delta \phi &= 0 \quad &&\mbox{ in } \Omega,\label{lap0}\\
    \phi &= u|_{\partial \Omega} \quad &&\mbox{ on } \partial \Omega.\label{boundary0}
\end{alignat}
It is straightforward to show that $u$ solves (\ref{lap0})-(\ref{boundary0}), however it holds only that $u \in H^{4/3}(\Omega)$. 

\begin{figure}[htb]
\begin{minipage}[t]{0.3\textwidth}\centering
\includegraphics[width=\textwidth]{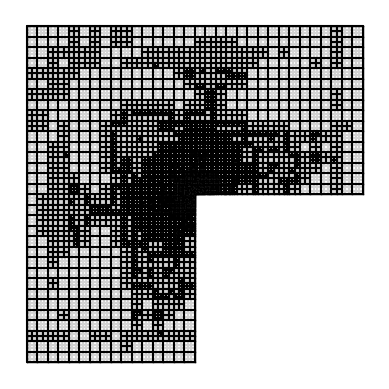}
(a) $p=2$
\end{minipage}
\hfill
\begin{minipage}[t]{0.3\textwidth}\centering
\includegraphics[width=\textwidth]{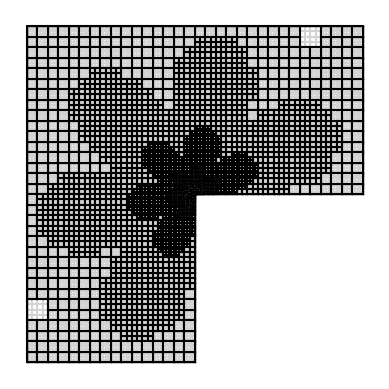}
(b) $p=3$
\end{minipage}
\hfill
\begin{minipage}[t]{0.3\textwidth}\centering
\includegraphics[width=\textwidth]{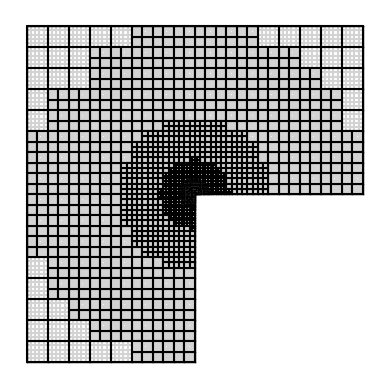}
(c) $p=4$
\end{minipage}
\caption{\label{fig:ex1:mesh} Final mesh for the first L-domain test case.}
\end{figure}

\begin{figure}[htb]
\begin{minipage}[t]{0.3\textwidth}\centering
\includegraphics[width=\textwidth]{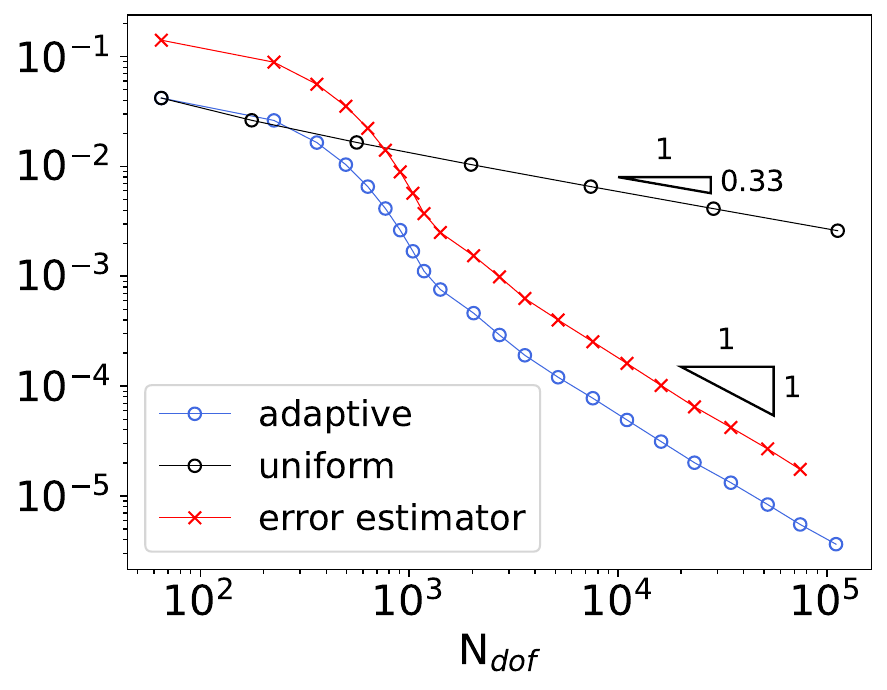}
(a) $p=2$
\end{minipage}
\hfill
\begin{minipage}[t]{0.3\textwidth}\centering
\includegraphics[width=\textwidth]{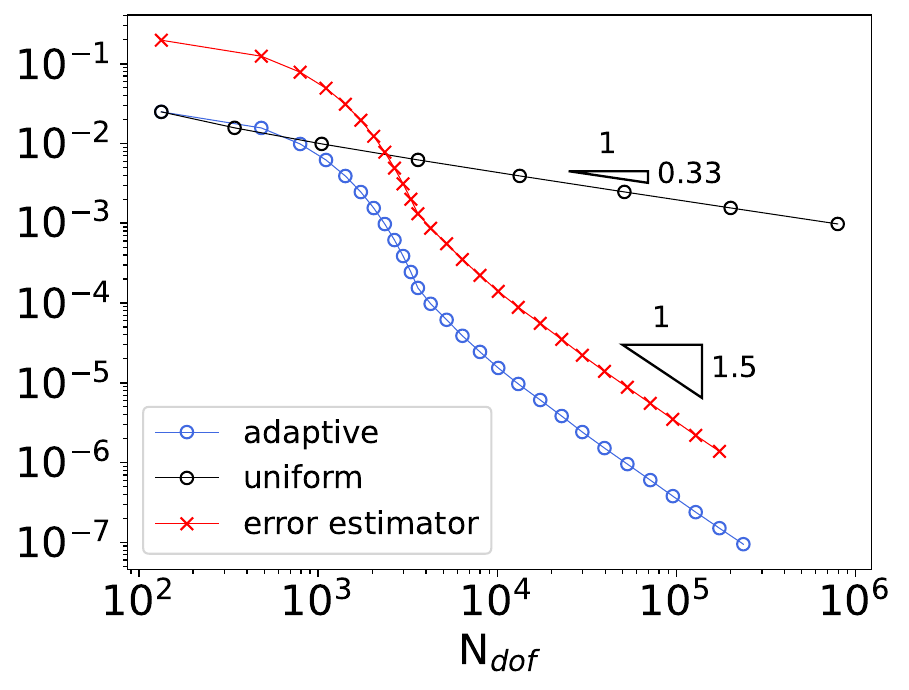}
(b) $p=3$
\end{minipage}
\hfill
\begin{minipage}[t]{0.3\textwidth}\centering
\includegraphics[width=\textwidth]{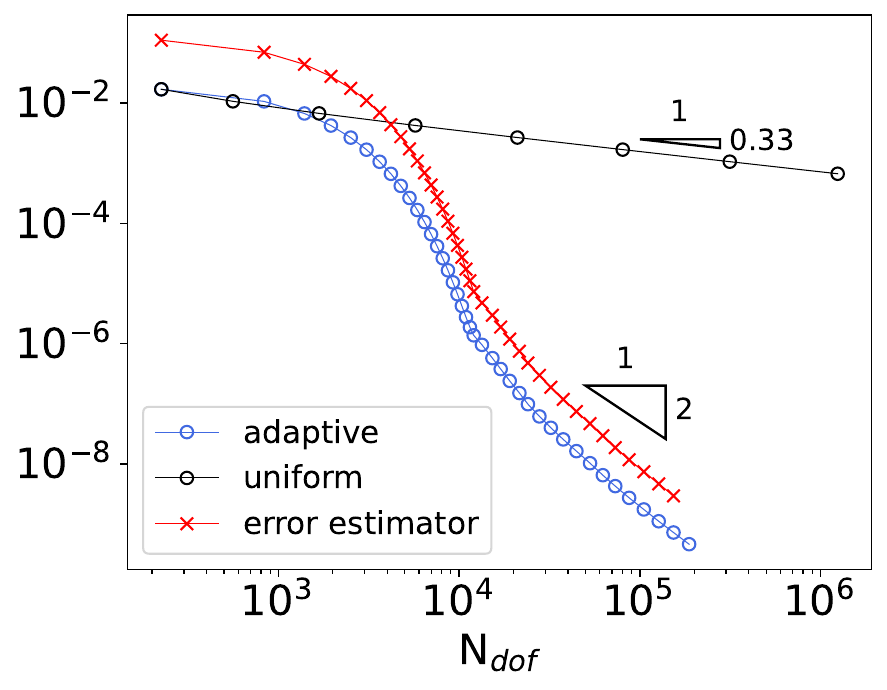}
(c) $p=4$
\end{minipage}
\caption{\label{fig:ex1:h1} $H^1$-error for the first L-domain test case.}
\end{figure}

\begin{figure}[htb]
\begin{minipage}[t]{0.3\textwidth}\centering
\includegraphics[width=\textwidth]{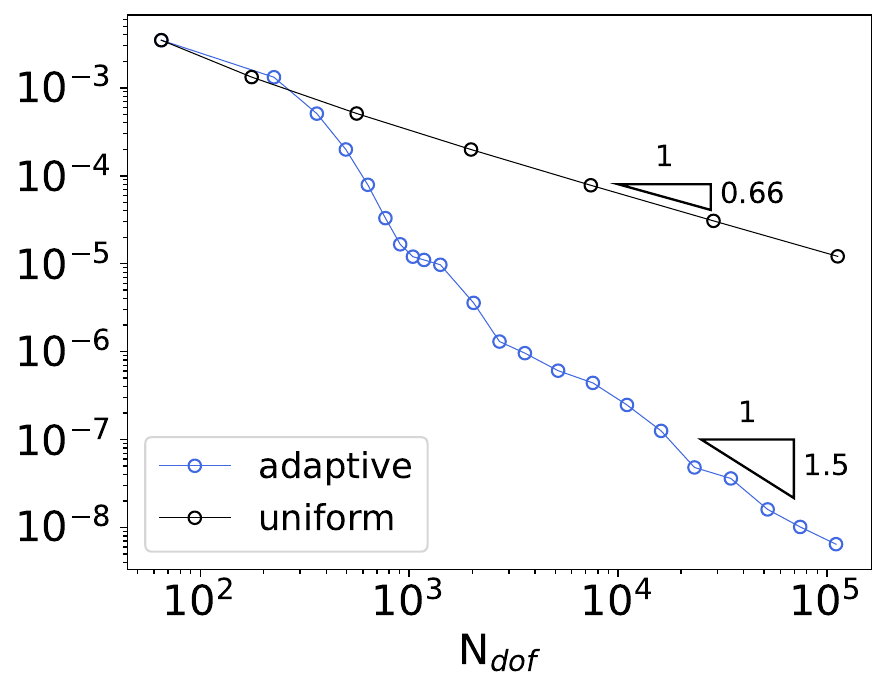}
(a) $p=2$
\end{minipage}
\hfill
\begin{minipage}[t]{0.3\textwidth}\centering
\includegraphics[width=\textwidth]{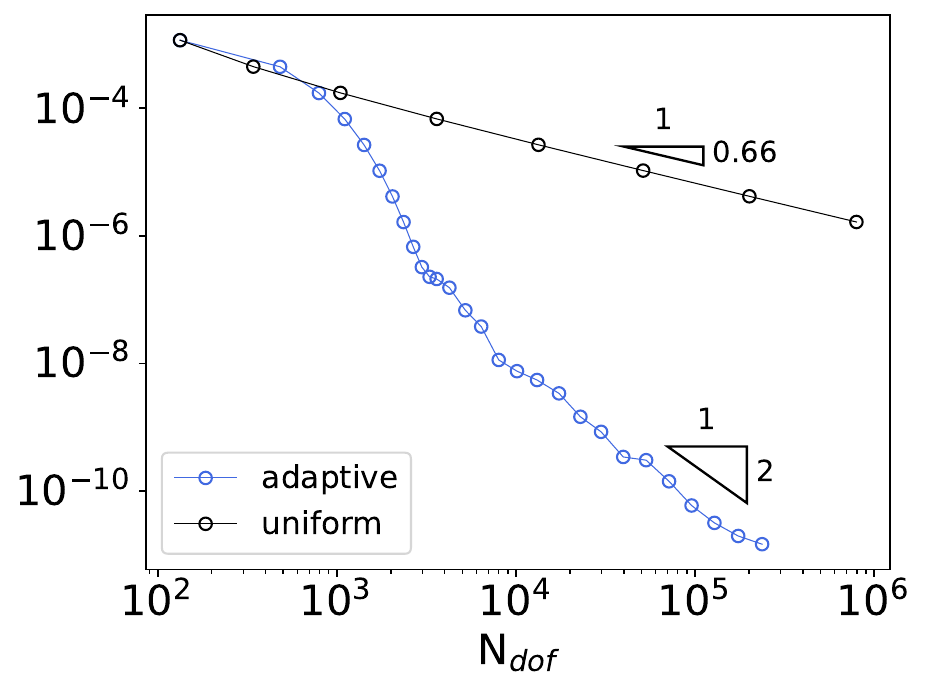}
(b) $p=3$
\end{minipage}
\hfill
\begin{minipage}[t]{0.3\textwidth}\centering
\includegraphics[width=\textwidth]{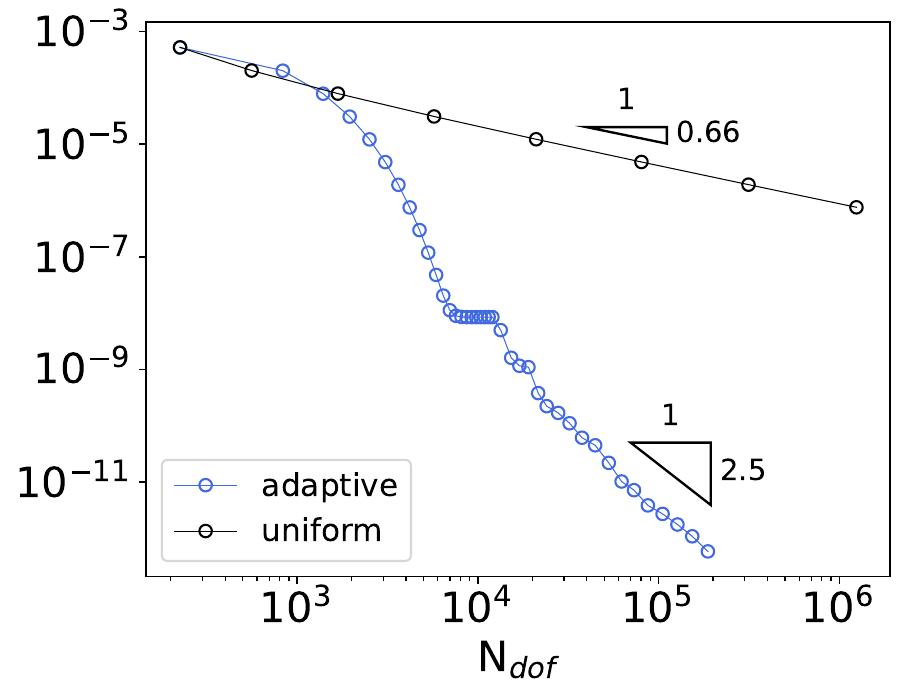}
(c) $p=4$
\end{minipage}
\caption{\label{fig:ex1:l2} $L^2$-error for the first L-domain test case.}
\end{figure}

We represent the L-domain with 3 conforming patches, where each patch is discretized by spline spaces of maximum smoothness and $p+1$ uniformly distributed knots in each coordinate, where $p$ is the spline degree. We further employ the residual error estimator~\eqref{res:err:est} on every patch. The final meshes that were obtained with adaptive refinement are depicted in Figure~\ref{fig:ex1:mesh}. In Figure~\ref{fig:ex1:h1}, the $H^1$-error and the value of the error estimator are presented. For the case of uniform refinement (black curve), one can see that the error decays like $h^{2/3} \eqsim {\mathrm N}_{dof}^{-1/3}$, which is consistent with the global regularity result $u\in H^{1+1/3}(\Omega)$. In case of adaptive refinement, one obtains for $p=2,3,4$ that -- after some faster pre-asymptotic behavior -- the $H^1$-error (blue curve) behaves like ${\mathrm N}_{dof}^{-2/p}$, which is the best one would expect for these spline degrees. The curve representing the error indicator (red curve) is parallel to that representing the $H^1$-error. Analogously, in Figure~\ref{fig:ex1:l2}, one can see that the $L^2$-error in case of uniform refinement decays like $h^{4/3} \eqsim {\mathrm N}_{dof}^{-2/3}$. Using adaptivity, one can recover the full rate ${\mathrm N}_{dof}^{-2/(p+1)}$.

Additionally we solved the Poisson problem on the L-domain with constant source and homogeneous boundary conditions, i.e. we solve
\begin{alignat*}{2}
    -\Delta \phi &= 1 \quad &&\mbox{ in } \Omega,\\
    \phi &= 0 \quad &&\mbox{ on } \partial \Omega.
\end{alignat*}
In this case, we do not know an exact solution in closed form. Since the right-hand side is smooth, the solution will be smooth in the interior as well. This does not extend to the boundary of the domain $\Omega$, particularly its corners. In the following we present the final mesh configuration for different polynomial degrees $p$.

\begin{figure}[htb]
\begin{minipage}[t]{0.3\textwidth}\centering
\includegraphics[width=\textwidth]{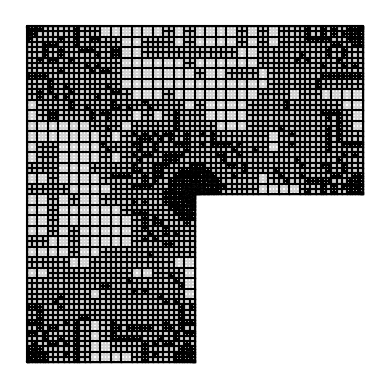}
(a) $p=2$
\end{minipage}
\hfill
\begin{minipage}[t]{0.3\textwidth}\centering
\includegraphics[width=\textwidth]{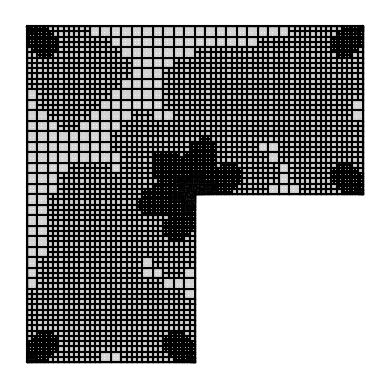}
(b) $p=3$
\end{minipage}
\hfill
\begin{minipage}[t]{0.3\textwidth}\centering
\includegraphics[width=\textwidth]{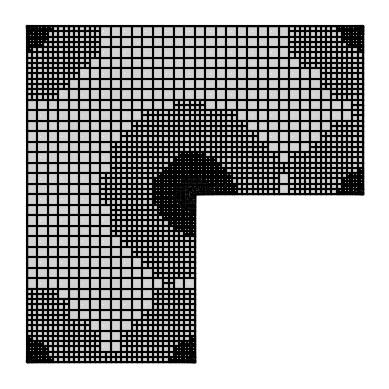}
(c) $p=4$
\end{minipage}
\caption{\label{fig:ex2:mesh} Final mesh for the second L-domain test case.}
\end{figure}

\begin{figure}[htb]
\begin{minipage}[t]{0.3\textwidth}\centering
\includegraphics[width=\textwidth]{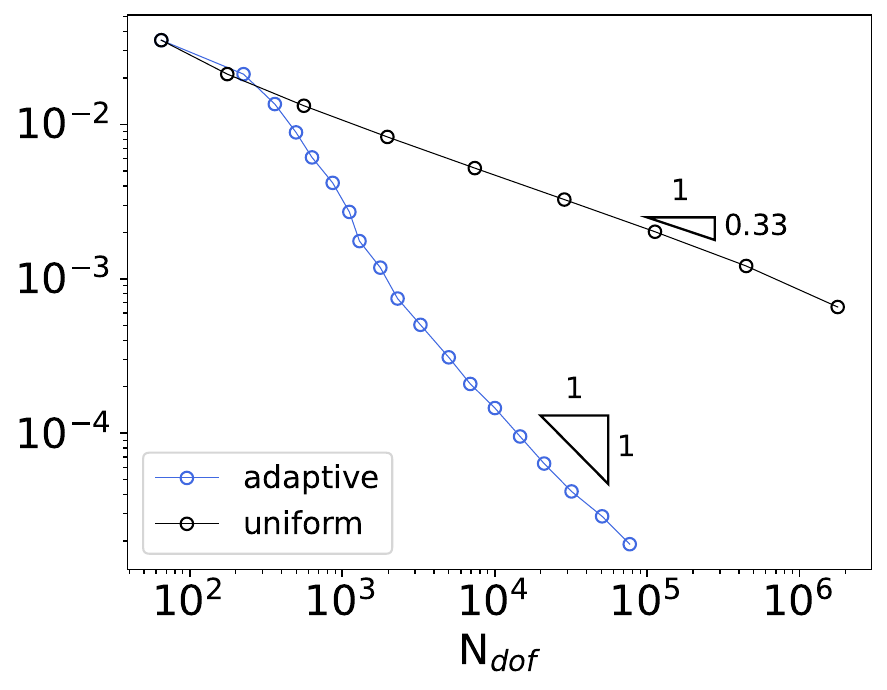}
(a) $p=2$
\end{minipage}
\hfill
\begin{minipage}[t]{0.3\textwidth}\centering
\includegraphics[width=\textwidth]{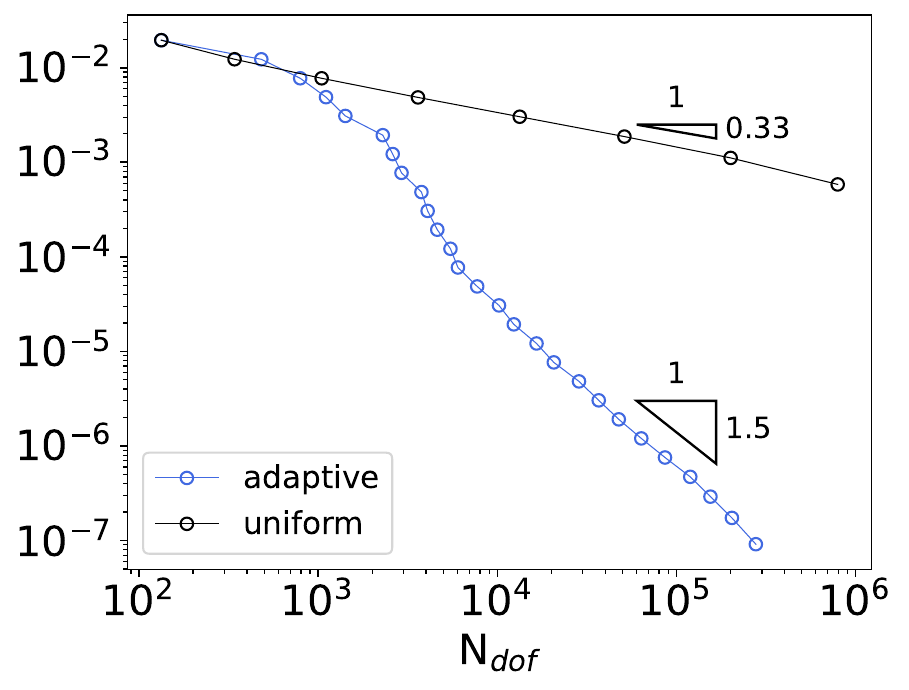}
(b) $p=3$
\end{minipage}
\hfill
\begin{minipage}[t]{0.3\textwidth}\centering
\includegraphics[width=\textwidth]{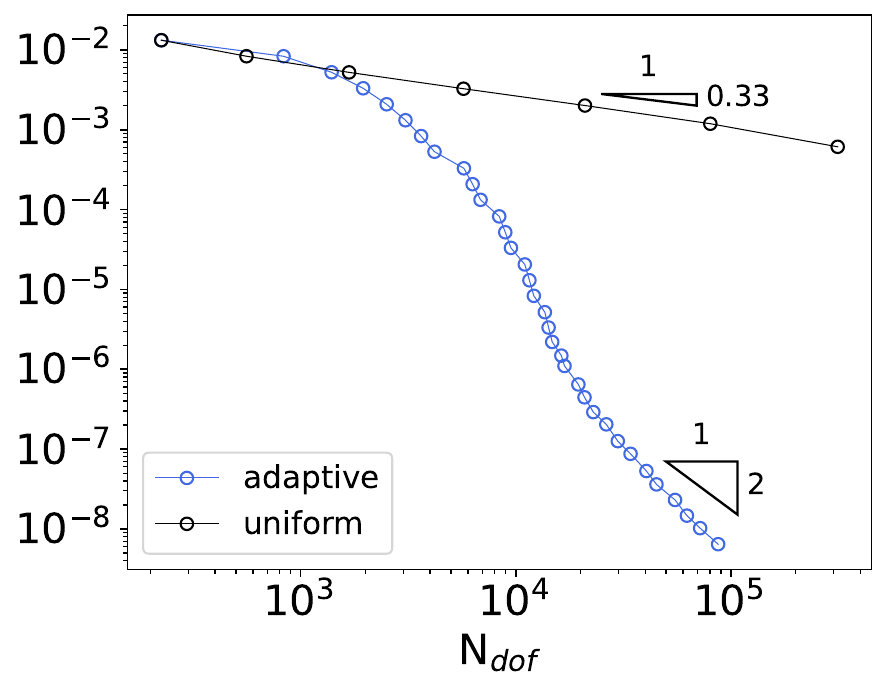}
(c) $p=4$
\end{minipage}
\caption{\label{fig:ex2:h1} $H^1$-error for the second L-domain test case.}
\end{figure}

\begin{figure}[htb]
\begin{minipage}[t]{0.3\textwidth}\centering
\includegraphics[width=\textwidth]{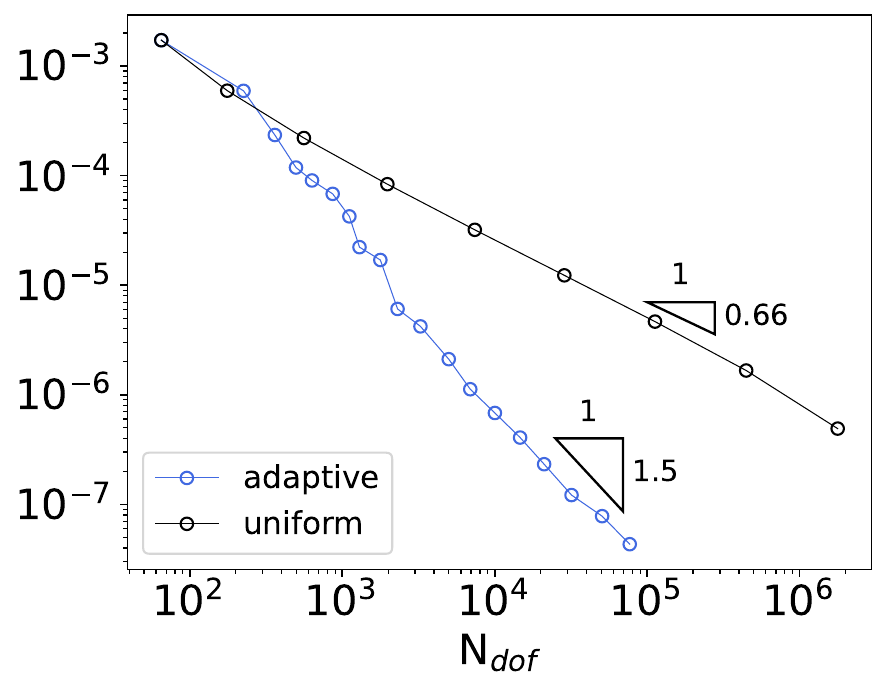}
(a) $p=2$
\end{minipage}
\hfill
\begin{minipage}[t]{0.3\textwidth}\centering
\includegraphics[width=\textwidth]{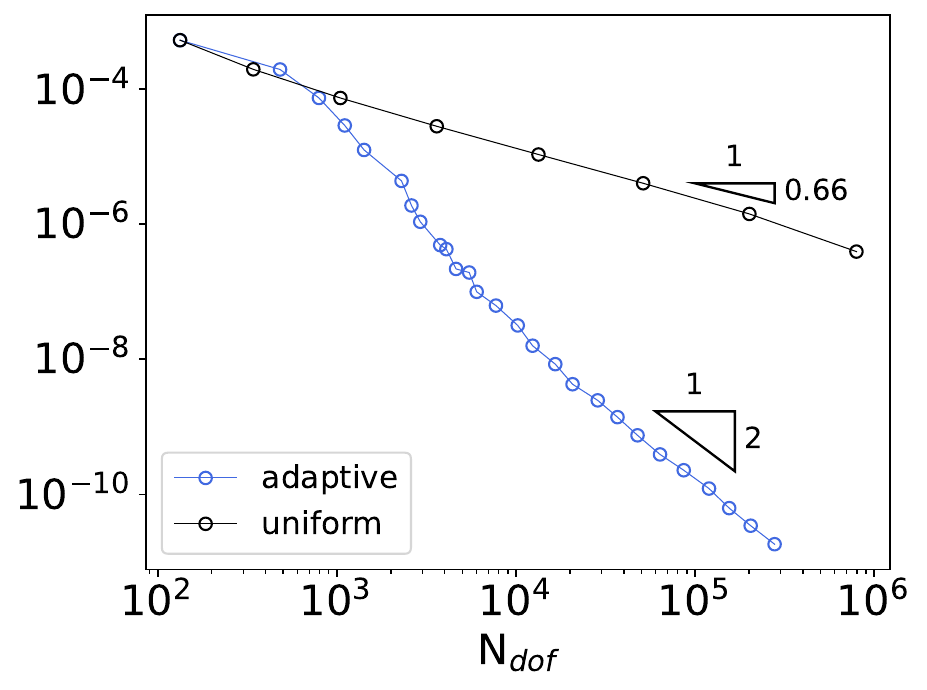}
(b) $p=3$
\end{minipage}
\hfill
\begin{minipage}[t]{0.3\textwidth}\centering
\includegraphics[width=\textwidth]{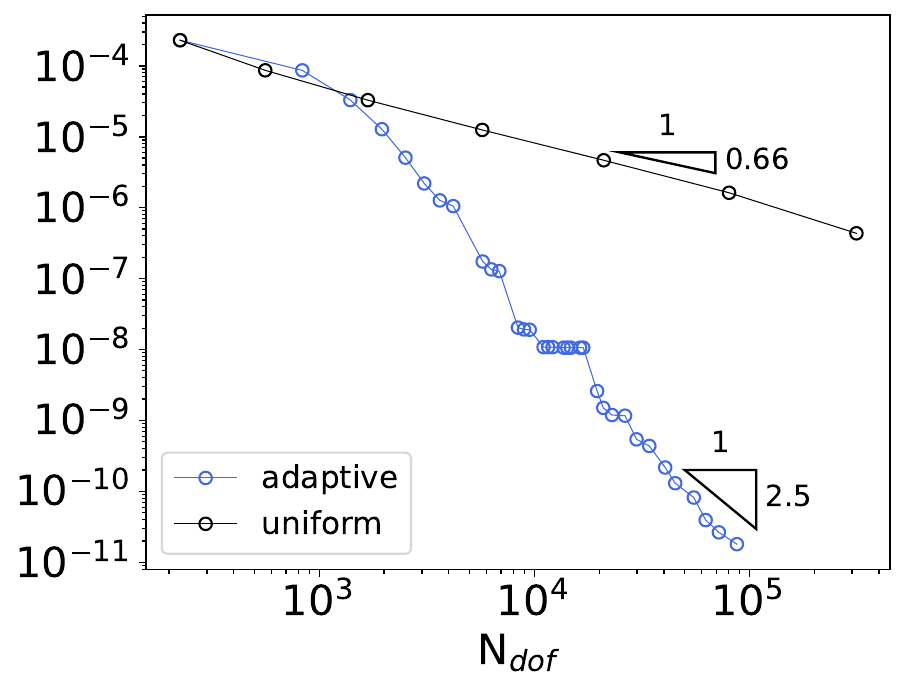}
(c) $p=4$
\end{minipage}
\caption{\label{fig:ex2:l2} $L^2$-error for the second L-domain test case.}
\end{figure}

For this example, Figures~\ref{fig:ex2:h1} and \ref{fig:ex2:l2} indicate that we obtain the same convergence behavior as for the first test case. In Figure~\ref{fig:ex2:mesh}, one can spot the difference to the first example. In the first example, the adaptive refinement algorithm has only refined towards the reentrant corner, since by construction the manifactured solution has a singularity only there. In the second example, one can see that also patches close to the regular corners of $\Omega$ are refined.

\subsection{Electric Motor}
We also apply the adaptive scheme to a (rather simplistic) electric motor, where the electromagnetic phenomena are prescribed by Maxwell's equations. In this example, we consider linear magneto-statics. For given permeability $\mu$, current density $J$ and permanent magnetization $M$, the magnetic field density $H$ and the magnetic flux density $B$ are given by
\begin{align*}
    \Curl\,H = J, \quad
    \Div\,B = 0, \quad
    B = \mu (H+M).
\end{align*} 

If the computational domain $\Omega \subset \RR^3$ is simply connected, the condition $\Div\,B=0$ implies the existence of a vector potential $A$ such that $B=\Curl\,A$. Assuming that $\Omega = \Omega_0 \times [0,L]$, $H=(H_1,H_2,0)^\top$, $B=(B_1,B_2,0)^\top$, $A=(0,0,u)^\top$, $J = (0,0,j)^\top$ and $M=(M_1,M_2,0)^\top$ the problem further simplifies to a scalar problem on the cross-section $\Omega_0$, depicted in Figure~\ref{fig:ex3:mesh} (a).
Here, we solve for the potential $u \in H_{D}^1(\Omega) := \{v \in H^1(\Omega):v|_{\Gamma_D} = 0\}$ that satisfies
\begin{alignat*}{2}
    -\Div(\mu^{-1} \nabla u) &= j + \Div\,M^\perp \quad &&\mbox{ in } \Omega_0, \\
    u&=0 \quad &&\mbox{ on } \Gamma_D, \\
    \mu^{-1} \nabla u + M^{\perp} &= 0 \quad &&\mbox{ on } \Gamma_N,
\end{alignat*}
in a weak sense, where $M^{\perp}:=(M_2,-M_1)^\top$.
The Neumann condition on $\Gamma_N$, colored in red in Figure~\ref{fig:ex3:mesh} (a), represents a periodicity assumption. The Dirichlet condition on $\Gamma_D = \partial\Omega_0 \setminus \Gamma_N$, colored in blue, represents an isolation in the magnetic flux.
We choose the current density $j=0$ on the whole domain, i.e, the electric motor is unpowered.
The permeability $\mu$ is defined piecewise depending on the material. For the air (white), we have $\mu_{\text{Air}} = 4 \pi\cdot10^{-7}$, for the iron (gray) $\mu_{\text{Iron}} = 204\pi \cdot 10^{-5}$, for the permanent magnet (yellow) $\mu_{\text{Mag}} = 4.344 \pi \cdot 10^{-7}$. So, we have jumps in the coefficient $\mu^{-1}$ of approximately the magnitude $10^4$. The permanent magnetization on the magnet area ($M$ vanishes everywhere else) is given by $M = \rho\,\mu^{-1}_{\text{Mag}}\,\bold{n}$, where $\rho=1.28$ is the magnetic remanence and $\bold{n}$ is the unit vector perpendicular to the centerline of the magnet. The direction of $\bold{n}$ is alternated between consecutive magnets.

\newcommand{\s}{3.4}
\begin{figure}[htb]
\begin{minipage}[t]{0.24\textwidth}\centering
\includegraphics[height=\s cm]{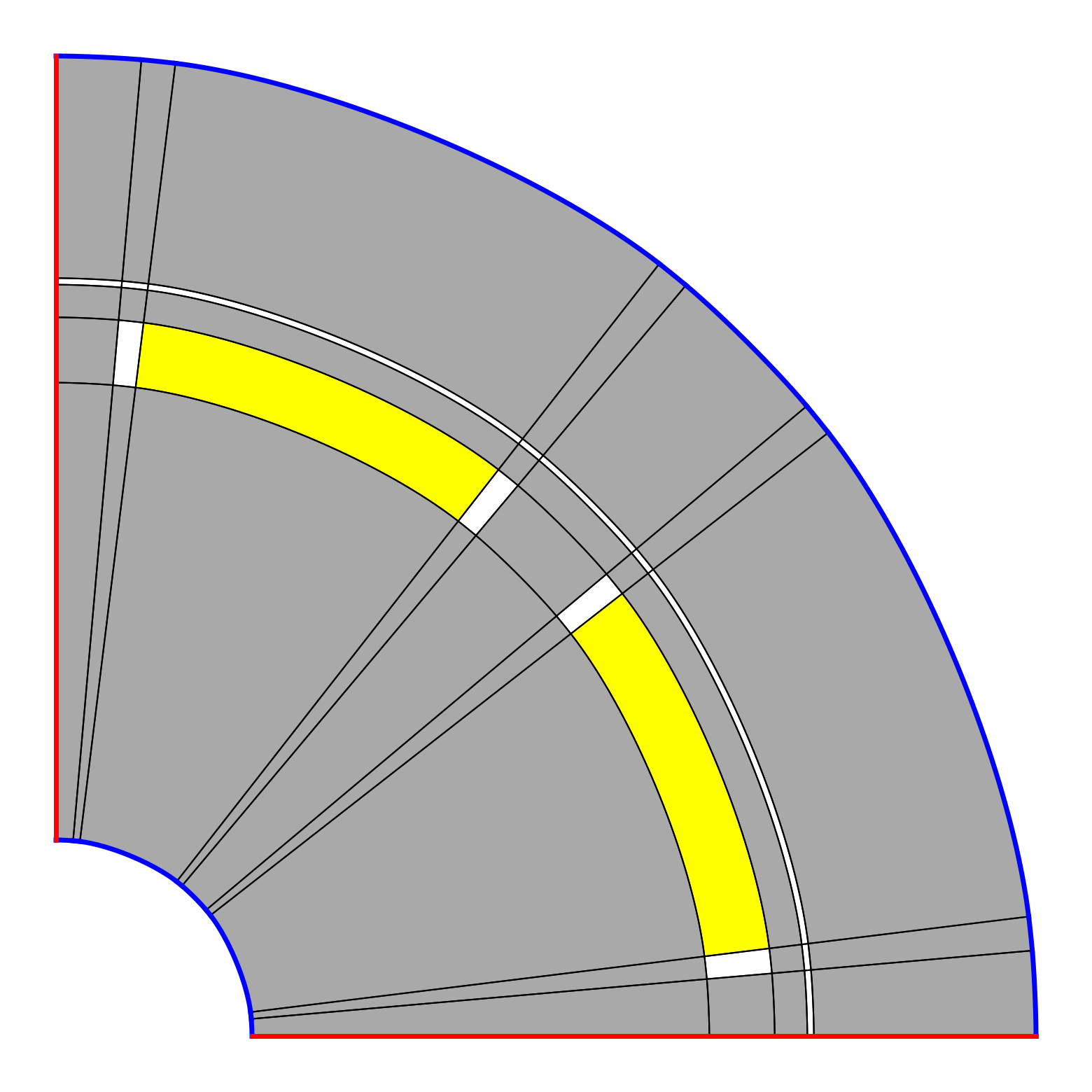}\\
(a) Material distribution
\end{minipage}
\hfill
\begin{minipage}[t]{0.24\textwidth}\centering
\includegraphics[height=\s cm]{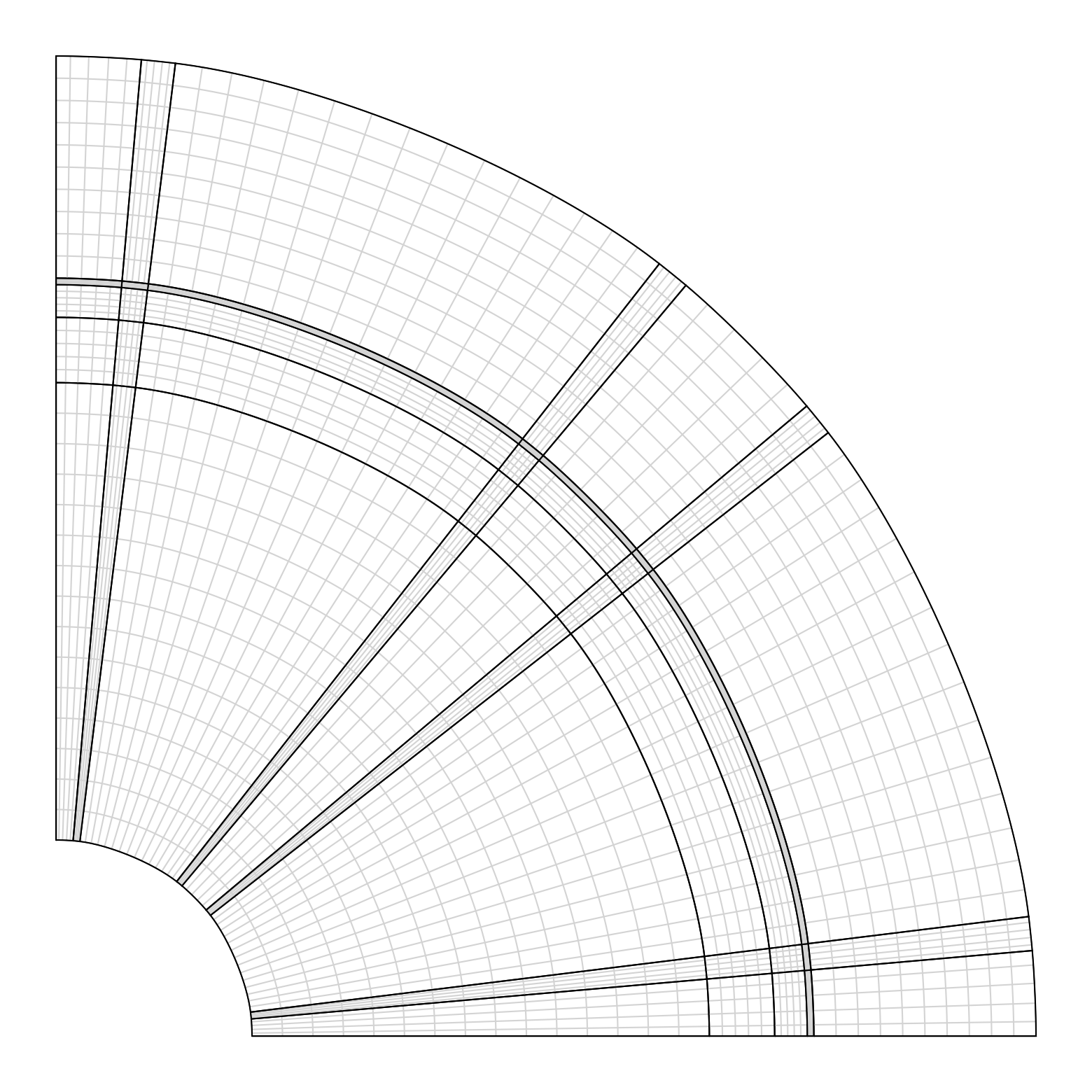}\\
(b) Initial patch configuration
\end{minipage}
\hfill
\begin{minipage}[t]{0.24\textwidth}\centering
\includegraphics[height=\s cm]{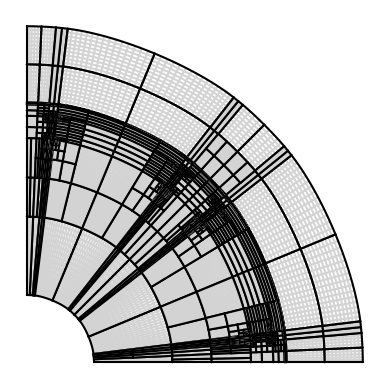}\\
(c) Adaptively refined mesh
\end{minipage}
\hfill
\begin{minipage}[t]{0.24\textwidth}\centering
\includegraphics[height=\s cm]{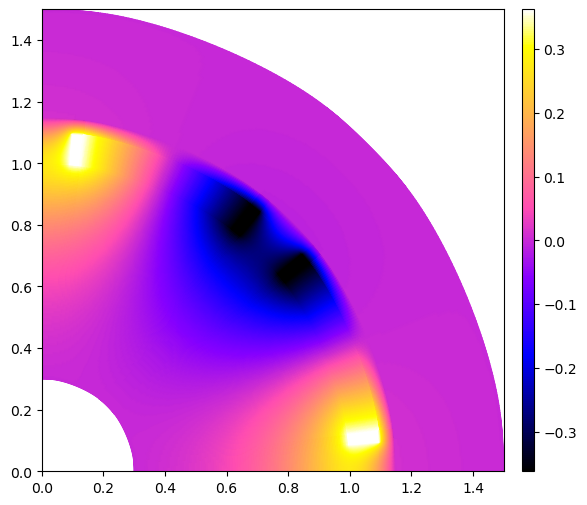}\\
(d) Vector potential
\end{minipage}
\caption{\label{fig:ex3:mesh} Electrical motor.}
\end{figure}

\begin{figure}[htb]
\begin{minipage}[t]{0.3\textwidth}\centering
\includegraphics[width=\textwidth]{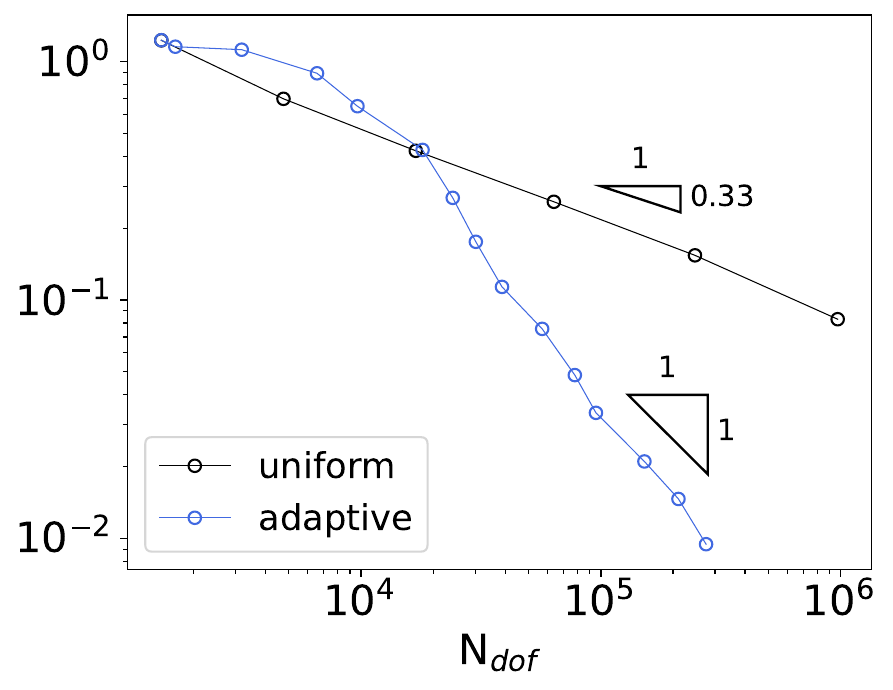}
(a) $p=2$
\end{minipage}
\hfill
\begin{minipage}[t]{0.3\textwidth}\centering
\includegraphics[width=\textwidth]{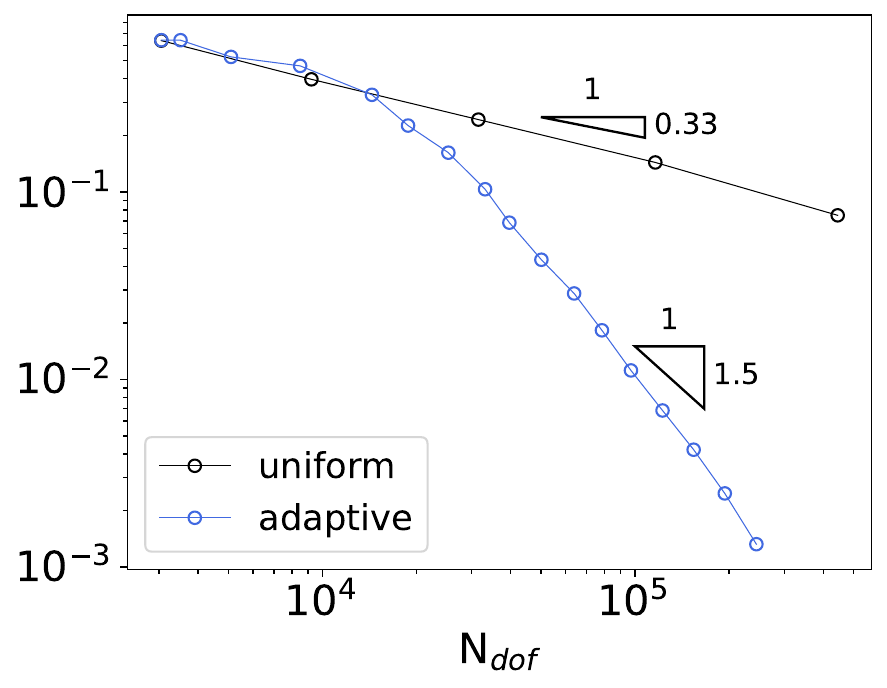}
(b) $p=3$
\end{minipage}
\hfill
\begin{minipage}[t]{0.3\textwidth}\centering
\includegraphics[width=\textwidth]{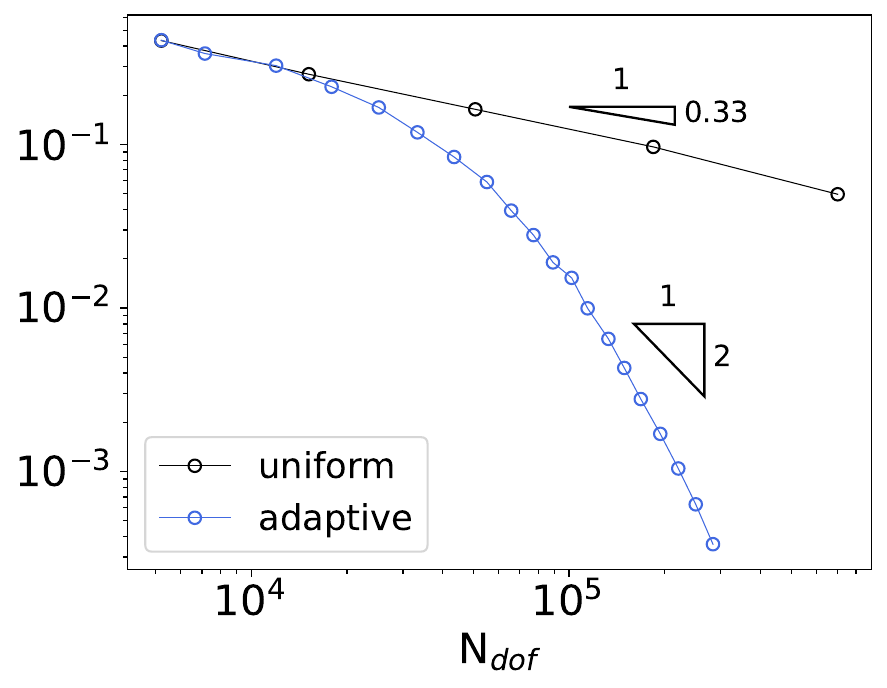}
(c) $p=4$
\end{minipage}
\caption{\label{fig:ex3:h1} Error in energy norm for the electrical motor.}
\end{figure}

\begin{figure}[htb]
\begin{minipage}[t]{0.3\textwidth}\centering
\includegraphics[width=\textwidth]{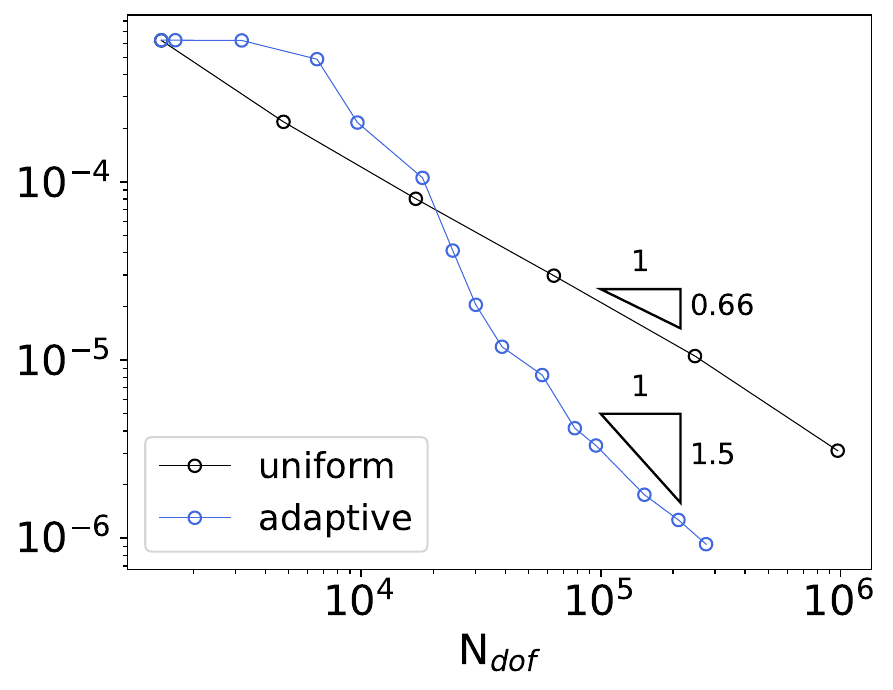}
(a) $p=2$
\end{minipage}
\hfill
\begin{minipage}[t]{0.3\textwidth}\centering
\includegraphics[width=\textwidth]{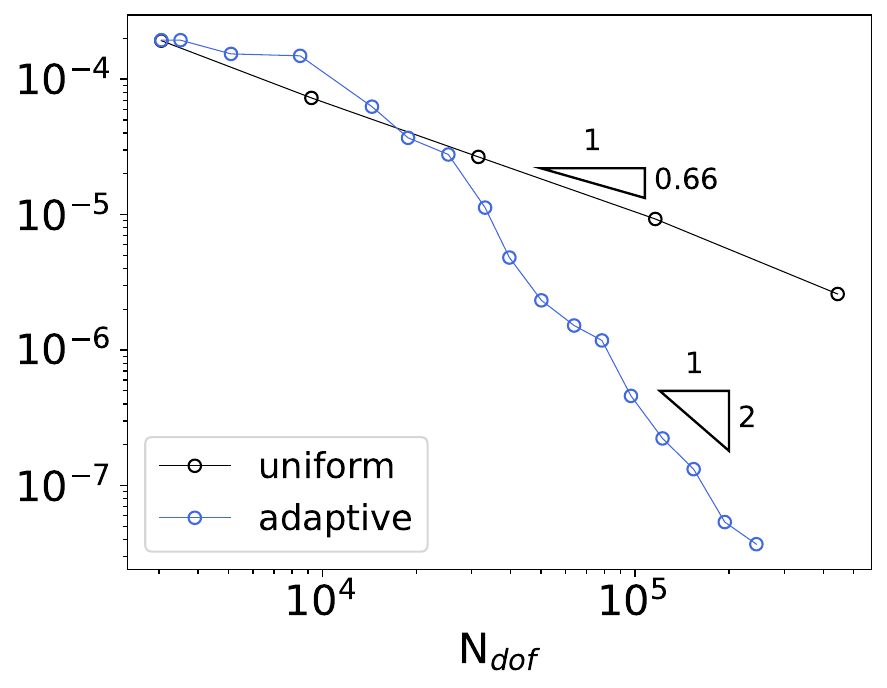}
(b) $p=3$
\end{minipage}
\hfill
\begin{minipage}[t]{0.3\textwidth}\centering
\includegraphics[width=\textwidth]{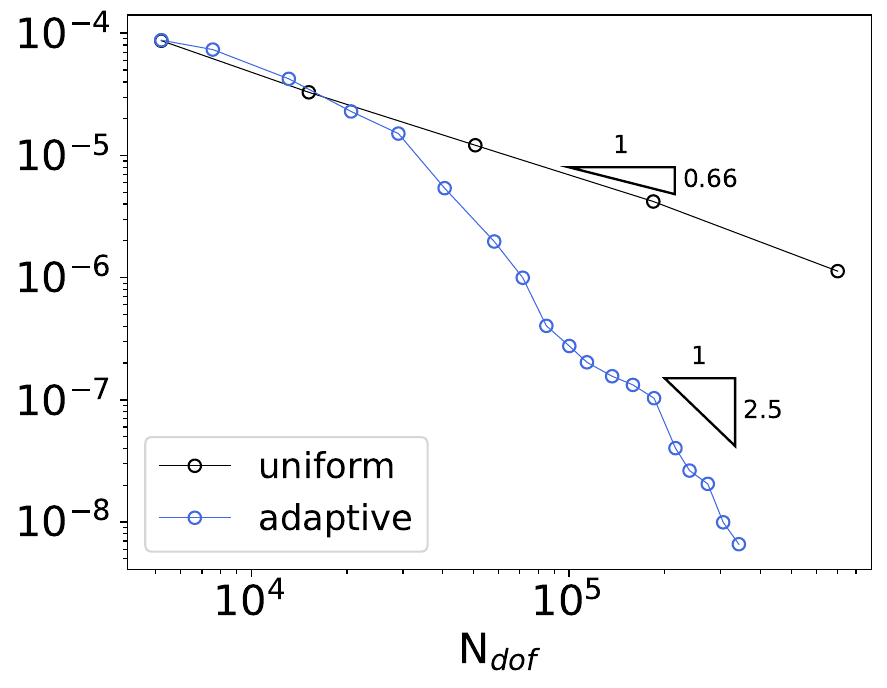}
(c) $p=4$
\end{minipage}
\caption{\label{fig:ex3:l2} $L^2$-error for the electrical motor.}
\end{figure}

The initial patch configuration used for discretization of the variational problem is depicted in Figure~\ref{fig:ex3:mesh} (b), where the patch boundaries are the black lines and the inner knot mesh is shown in light gray. One can see that the domain does not have reentrant corners that would reduce the smoothness of the solution. However, the material parameters vary significantly, which again reduces the smoothness of the solution. The solution is depicted in Figure~\ref{fig:ex3:mesh} (d); it was computed on the adaptive mesh seen in Figure~\ref{fig:ex3:mesh} (c). The Figures~\ref{fig:ex3:h1} and \ref{fig:ex3:l2} show that the convergence rates are again reduced if we use uniform refinement. Since the singularities are caused by 90-degree corners adjacent to the permanent magnets, the error in the energy norm converges like $h^{2/3}\eqsim \mathrm{N}_{dof}^{-1/3}$ and the $L^2$ error like $h^{4/3}\eqsim \mathrm{N}_{dof}^{-2/3}$. The full rate that would be expected for the spline degree $p=2$ is again obtained by using the adaptive approach. For higher spline degrees ($p=3,4$) we even seem to observe superconvergent rates, however, this might still be in the pre-asymptotic regime.

\section{Conclusions}\label{sec:7}

We proposed a new approach to employ adaptive mesh refinement in the framework of multi-patch Isogeometric Analysis that allows to re-use computational frameworks for the handling of multi-patch geometries that are already well-established. Specifically, on each patch, the tensor-product structure of the B-spline basis is retained.
Since our approach increases the number of patches, we introduce additional interfaces that are only $C^0$ continuous. This, however, only marginally increases the number of degrees of freedom compared to alternative methods (like HB- or THB-splines or T-splines). Since it reduces the overlap between the supports of basis functions of different mesh levels, our approach is both simpler and can be implemented in a more efficient way. Through the subdivision of local patches into sub-patches, hanging nodes are emerging, which leads to non-matching discretizations along common interfaces. These discretizations are coupled through constraints on each edge, yielding a $H^1$-conforming function space.

For practical computations, a basis of the overall space might be desired. We have developed an algorithm that uses these constraints and constructs a basis for the whole spline space. Since it is purely matrix based, it can be extended to 3D effortlessly. The constructed global basis retains all important properties of the local spline bases, like the non-negativity of the basis functions, that they form a partition of unity and have local support.

Additionally we gave an a-priori approximation error analysis by constructing a suitable quasi-interpolation operator. For the high regularity case ($u \in H^2(\Omega)$), we could even show $p$-robustness. An analogous result is obtained for the low regularity case ($u \in H^{1+q}(\Omega)$ with $q\in (0,1)$), however its dependence on the parameters seems not to be sharp.
We further gave some numerical examples confirming applicability of adaptive schemes with our approach of patch-subdivision.

\printbibliography
\end{document}